\documentclass[12pt, a4paper,reqno]{amsart}
\usepackage[margin=1in]{geometry}
\usepackage{amssymb}
\usepackage{enumerate}
\usepackage{hyperref}
\usepackage{doi,url}
\usepackage{dsfont}
\usepackage{mathabx}
\usepackage{booktabs,tabularx}
\usepackage[dvipsnames]{xcolor}
\usepackage{subfigure}

\definecolor{purple}{rgb}{.5,0,1}
\definecolor{orange}{rgb}{1,.5,0}
\definecolor{pink}{rgb}{1,0,.5}

\usepackage{tikz}
\usetikzlibrary{shapes,arrows}
\tikzset{
  block/.style    = {draw, thick, rectangle, minimum height = 3em,
    minimum width = 3em},
  sum/.style      = {draw, thick, rectangle, minimum height = 3em,
    minimum width = 3em},
      input/.style    = {coordinate}, 
  output/.style   = {coordinate} }
\newcommand{\sumazero}{\Large$S_0$}
\newcommand{\sumaone}{\Large$S_1$}

\newcommand{\sumak}{\Large$S_k$}
\newcommand{\suman}{\Large$S_n$}
\newcommand{\sumai}{\Large$I$}\usepackage{textcomp}
\definecolor{auburn}{rgb}{0.43, 0.21, 0.1}\definecolor{ultramarine}{RGB}{0,70,10}
\definecolor{bluemunsell}{rgb}{0.0, 0.5, 0.69}
\definecolor{applegreen}{rgb}{0.55, 0.71, 0.0}
\definecolor{aquamarine}{rgb}{0.5, 1.0, 0.83}

\definecolor{purple}{rgb}{.5,0,1}
\definecolor{orange}{rgb}{1,.5,0}
\definecolor{pink}{rgb}{1,0,.5}

\usepackage{fancyvrb}
\usepackage{amsmath,amsthm}
\usepackage{enumitem}
\usepackage{thmtools}
\usepackage[style=apalike,style=numeric,maxbibnames=99,firstinits=true,citestyle=numeric,url=true,doi=true,
backend=bibtex
]
{biblatex}
\renewbibmacro{in:}{}
\AtEveryBibitem{\clearfield{number}}
\DeclareFieldFormat*{volume}{\mkbibbold{#1}}
\DeclareFieldFormat[article]{pages}{#1}

\numberwithin{equation}{section}
\newtheorem{thm}{Theorem}[section]

\newtheorem{lem}[thm]{Lemma}

\newtheorem{rem}[thm]{Remark}

\newenvironment{acknowledgement}{\emph{Acknowledgement.}}

\newtheorem*{theorem*}{Theorem}
\newtheorem*{lemma*}{Lemma}
\newtheorem*{remark*}{Remark}

\newcommand\C{\mathbb C}

\newcommand\beq{\begin{equation}}
\newcommand\eeq{\end{equation}}
\newcommand\be{\begin{equation}\begin{aligned}}
\newcommand\ee{\end{aligned}\end{equation}}

\newcommand{\abs}[1]{\left\lvert #1 \right\rvert}
\newcommand{\norm}[1]{\left\lVert #1 \right\rVert}

\newcommand{\pa}[1]{\left( #1 \right)}

\begin{document}

\title[A Perturbative Approach to the Analysis of Many-Compartment Models]{A Perturbative Approach to the Analysis of Many-Compartment Models Characterized by the Presence of Waning Immunity}

\author{Shoshana Elgart}
\address[S. Elgart]{Blacksburg, VA, 24060, USA}
 \email{shosha.elgart@gmail.com}



\begin{abstract} 
The waning of immunity after recovery or vaccination is a major factor accounting for the severity and prolonged duration of an array of epidemics, ranging from COVID-19 to diphtheria and pertussis. To study the effectiveness of different immunity level-based vaccination schemes in mitigating the impact of waning immunity, we construct epidemiological models that mimic the latter's effect. The total susceptible population is divided into an arbitrarily large number of discrete compartments with varying levels of disease immunity. We then vaccinate various compartments within this framework, comparing the value of $R_0$ and the equilibria locations for our systems to determine an optimal immunization scheme under natural constraints. Relying on perturbative analysis, we establish a number of results concerning the location, existence, and uniqueness of the system's endemic equilibria, as well as results on disease-free equilibria. In addition, we numerically simulate the dynamics associated with our model in the case of pertussis in Canada, fitting our model to available time-series data. Our analytical results are applicable to a wide range of systems composed of arbitrarily many ODEs.
\end{abstract}
%

\maketitle

\section{Introduction}\label{sec1} 

In this work, we introduce and analyze the behavior of many-compartment epide\-miological models depicting waning immunity boosted by an optimized immunization scheme. For a number of infectious diseases, including pertussis, increased levels of immune memory after vaccination or recovery are known to wane significantly with time. Controlling such epidemics in the long-term often requires the introduction of effective booster vaccination schemes, among additional public health measures.

The analysis of models developed through mathematical epidemiology has been used to inform policy in the public health sphere beginning with the 18th-century mathematician Daniel Bernoulli's smallpox inoculation model, whose results emphasized the necessity of variolation in controlling recurring smallpox epidemics (\cite{Bacaer}). Throughout the COVID-19 pandemic, infectious disease modeling has assumed a central role in the decision-making process of multiple government bodies (\cite{MB}; \cite{James}), informing the use of travel restrictions and the length of lockdowns, among multiple other mitigation strategies.

The SIR (susceptible-infected-recovered) model is the cornerstone of contemporary mean-field compartmental models, including those used in this work. Its origin can be traced back to the seminal work of \cite{Kermack1}. In particular, the SIR model is a reduced version of the general ODE age-of-infection model introduced in that paper, where transmissibility as well as recovery and death rates are fixed at constant values over the duration of an individual’s infection. The SIR model has played a crucial role in the field's later advances, as both deterministic and statistical depictions of infectious diseases are often constructed along the same principles as the SIR model. 

To briefly describe these principles, the SIR model divides a population affected by an epidemic into three compartments $S$, $I$, and $R$, consisting respectively of susceptible, infectious, and recovered individuals. Throughout the epidemic, susceptible individuals can come into contact with infectious individuals and contract the disease, while infectious individuals can recover after a given period, resulting in movement between compartments. This variation in compartment size over time can then be represented by a system of ordinary differential equations (ODEs). Demography is often introduced into the SIR model with the addition of a parameter representing the birth and death rates for each compartment. 

Multiple SIR-based models have been developed to depict the waning and boosting of immunity at both within-host and population scales. We will discuss a subset of these texts most relevant to the current paper (referring the reader to \cite{BRM} or similar articles for a comprehensive summary of both older and more modern immunity-based models).

Many models (see e.g., \cite{AG}, \cite{Arinaminpathy}, the delay differential equation model \cite{BRP}) set a constant number of total compartments: The work \cite{AG}, for instance, introduces a seven-compartment ODE COVID-19 model that represents waning immunity by differentiating between ``susceptible" (here referring to never infected), ``recovered", and ``susceptible but previously infected" individuals. Individuals in the second compartment are assumed to have full immunity, but waning immunity propels them into the third compartment, from where they can once again contract the disease. 

To enable comparison between these works and ours, we will refer to models which, like the above, can be characterized by a given constant number of compartments as \textit{few-compartment} models. 

In contrast, other models (see e.g., \cite{Lavine}, \cite{B}, \cite{HK}, \cite{C}) use an arbitrarily large number \textit{n} of compartments (in contrast to the few-compartment models previously discussed, we will refer to such models as \textit{many-compartment}). In \cite{Lavine}, the SIRS model is modified to contain a waning immunity compartment $W$ in addition to the recovered compartment $R$. Individuals in $R$ are fully or at least mostly immune, but individuals in $W$ are less immune and will eventually reach full susceptibility (joining compartment $S$ again). The work \cite{C} studies both a partial differential equation (PDE) and a discretized ODE model, depicting waning immunity and vaccination schemes for pertussis using five susceptible compartments through which individuals can pass through as their immunity wanes from full to minimal. The incorporation of multiple age classes, among other factors, brings both \cite{C} and \cite{Lavine} to the many-compartment class within our nomenclature. \cite{HK} includes an immuno-epidemiological model, depicting the individual within-host dynamics of an infectious disease in addition to the dynamics of the population itself. Here, the disease resistance of a population's susceptibles is differentiated via an arbitrary number of discrete sub-compartments, each varying in immune memory level, located inside a larger ODE SEIR model. The general model in \cite{B}, on the other hand, incorporates a continuous spectrum of immunity ranging from minimal to maximal, through which individuals are able to pass through with time. 

We now turn our attention to the ways in which these few and many-compartment models are analyzed. Many of the works considering few-compartment models determine the location and stability of both disease-free and endemic equilibria analytically, often supplementing rigorous analysis with a numerical simulation. (An example of this is the four-compartment work \cite{OSBPR}, which addresses each of the above for a SIRWS model, among other bifurcation-related results.) Disease-free equilibria are often studied using the next-generation matrix method, introduced in \cite{DW}, which yields the basic reproduction number $R_0$ for compartmental models satisfying a set of biologically-realistic assumptions. The many-compartment model \cite{B} is analyzed using combinations of multiple methods that determine location and stability for the DFE (among other results), and \cite{HK} studies endemic behavior through a numerical simulation. \cite{Lavine} also studies model dynamics numerically, though this work is not explicitly concerned with equilibria-related results. In \cite{C}, determining the location of the endemic equilibrium relies on prior knowledge about the steady state value for one of the age groups, which allows the analysis to be reduced to that of linear systems.

One crucial aspect of model analysis that should be noted at this stage is that the majority of nonlinear epidemiological model systems become linear when the interaction between healthy and infected individuals is removed. As a result, finding the location and determining the local stability of a system's DFE - where the population has no infectious individuals - is a linear problem that is usually solved with straightforward matrix analytical methods. On the other hand, since the original system is nonlinear, determining the precise locations of endemic equilibria is in general a much harder problem that cannot always be solved in closed form when the number of non-linearly interacting compartments is large. This complexity issue is noted in multiple works, including \cite{B}, whose integral equation-based approach cannot be applied to the endemic equilibrium. 

Waning immunity is a continuous process that, in the mean-field approximation, is most naturally modeled by non-linear PDEs. A somewhat less natural, though still realistic, depiction involves discretizing the waning immunity process into one with an arbitrarily large number of stages, which can be represented through many-compartment models such as those in \cite{B}. Due to the complexity issue described above, analyzing the detailed dynamical properties of the resulting nonlinear systems at the endemic equilibrium is a difficult problem that may also hinder numerical study for such systems, as many stable numerical approaches are derived from analytical work. 

Our main observation in the paper is that, in the context of waning immunity, one can often view a fully many-compartment model as a relatively weak perturbation of a few-compartment system (which is, however, still fully nonlinear). Thus, to address this problem, we devise a perturbative approach to yield explicit results for the location (and other characteristics) of the endemic equilibrium for such many-compartment model systems, and apply this approach to many-compartment models containing an discretized representation of immunity boosting and waning. These models are also adapted in Section \ref{Epidem} for the study of pertussis in Canada. 
 
We also aim to optimize the precise levels of susceptibility at which individuals should receive booster vaccinations. The rationale underlying the latter problem can be described as follows:  If there is significant vaccination coverage for individuals at sufficiently-high immune memory levels $\{k_i\}$, intervals between successive vaccinations will be short, which requires more supplies and may in a practical sense be impossible for the general population. On the other hand, allowing individuals to be vaccinated only at immunity levels $\{k_i\}$ that are low relative to some maximal immunity level means that individuals awaiting immunization become over time more and more susceptible to the virus, increasing their likelihood of contracting the disease in the meantime.

Our perturbation theory is supplemented with other analytical techniques, allowing us to bound the location of endemic equilibria for each of our model systems within small intervals and determine the equilibria's existence and uniqueness, while still retaining a relatively natural model of waning immunity. We also obtain a number of results pertaining to the local asymptotic stability of both endemic equilibria and the DFE. 

The remainder of the paper is organized as follows: In Section \ref{modform}, we describe the models used, introducing model structure and parameters. In Section \ref{secIII}, we analyze the location and asymptotic stability of infection-free and endemic equilibria. Our analytical results show that if vaccinations are restricted to individuals outside of the least-immune group, the value of $R_0$ for the system is independent of the rate of vaccination and the vaccine coverage. If we, however, enable vaccination for only the least-immune group at some coverage,  $R_0$ is strictly lowered, and moreover becomes a decreasing function of the vaccination rate and coverage. (For these results, we assume that the transmission rates for the most immune individuals are strictly lower than those for the most susceptible individuals.)

We also show that model's $R_0$ is an accurate indicator for the endemic behavior of the system if the rate of waning immunity is sufficiently small. Specifically, if $R_0 < 1$, no biologically-realistic endemic equilibrium exists for the system, and the sole realistic fixed point is the stable DFE. If $R_0 > 1$, however, a unique biologically-realistic endemic equilibrium exists, and is moreover locally asymptotically stable. 

These analytical results are followed-up by numerical results as detailed in Section \ref{Epidem}. We discuss the implications of  our findings in broader detail in Section \ref{sec7}. 

\section{Model Formulation}\label{modform}

\subsection{Description of Models and Notation}

In this section, we outline the model used (see \eqref{eq:3c} below for the corresponding model equations): The central aspect of this depiction of waning immunity involves splitting the $S$ compartment in the SIR model into (sub-)compartments $S_0,S_1,\ldots,S_n$ for some $n\in \mathbb N$, where each $S_k$ is characterized by a different level of susceptibility to the disease and thus subject to a different transmission rate $\beta_k$. This is done in order to depict a (reasonably realistic) gradual loss of immunity, while still retaining a model structure more amenable to analytical techniques than corresponding PDE models. 

While there could be many factors that can affect the immunity of a susceptible individual and the probability of the latter contracting the disease, on the mean-field level  immunity levels depend only on the time elapsed between the moment of recovery or immunization and present time. To simplify the analysis, it is assumed that the infection rates of just recovered and freshly immunized individuals are the same. As such, both categories are placed into the most-immune ($S_0$) compartment. 

As time progresses and immune memory erodes, an individual in a given susceptible compartment $S_k$ will move to the adjacent compartment $S_{k+1}$ for $k<n$ at the rate $\delta$. We note that, by construction, the transmission rates satisfy the relations $\beta_0\le \beta_1\le\ldots\le \beta_n$. 

We investigate the effect of immunizing a proportion of individuals in each susceptible compartment at compartment-dependent coverage values $\{p_i\}_{i = 0}^{n}$. (For all  $0 \le i \le n$, the proportion of the population vaccinated in a compartment $S_i$ is $p_{i}S_{i}$, where $0 \le p_{i} \le 1$. We set $p_0 = 0$, i.e., individuals with maximal immunity are not vaccinated. Immunized individuals are transferred to $S_0$ at a rate $\omega$ (which is assumed to be compartment-independent).  We will assume that newborns are born into the fully susceptible compartment $S_n$ (this assumption can be relaxed by suitably modifying the model's parameters). The following parameters are used (where all rates below are in the units of years$^{-1}$, the parameter values are nonnegative, the values of $\mu, r$ are positive, and $1/\delta$ is the amount of time in years required to pass from compartment $S_m$, $0 \le m < n$, to compartment $S_{m+1}$).
\\ \\ 

\begin{tabular}{
  |p{\dimexpr.17\linewidth-2\tabcolsep-1.3333\arrayrulewidth}
  |p{\dimexpr.8\linewidth-2\tabcolsep-1.3333\arrayrulewidth}|
  }
  \hline
   {\bfseries Parameters}     & {\bfseries Parameter Values}        \\ \hline
$\delta$ & The approximate rate of waning immunity  \\ \hline
$\beta_i$ &The transmission rate for the compartment $S_i$ \\ \hline
$r$ & The recovery rate for infected individuals\\ \hline
 $\mu$ & The birth and death rates, presumed to be equal 
\\ \hline
 $\omega$ & The vaccination rate for any compartment $S_k$ with nonzero coverage \\ \hline
  $p_i$ & The coverage rate for the compartment $S_i$ \\ \hline
\end{tabular}

\newpage

Then our primary ODE model system takes the form

\be\label{eq:3c}
\begin{aligned}
&\frac{dS_0}{dt} =  \omega \sum_{i=1}^{n}p_i
 S_i- \delta S_0 + rI - \beta_0 I S_0-\mu S_0; \\
&\frac{dS_i}{dt}  = -\omega p_iS_i+\delta ((1-p_{i-1})S_{i-1} - (1-p_i)S_i) - \beta_i I S_i-\mu S_i, \hspace{3pt} i\neq \{0,n\}; \\
&\frac{dS_n}{dt} = \mu -\omega p_nS_n+\delta(1-p_{n-1}) S_{n-1} - \beta_n I S_n-\mu S_n;\\
&\frac{dI}{dt} = I\sum_{i=0}^n \beta_i S_i  - rI-\mu I.
\end{aligned}
\ee
A schematic diagram of this model is provided in Figure \ref{fig}.
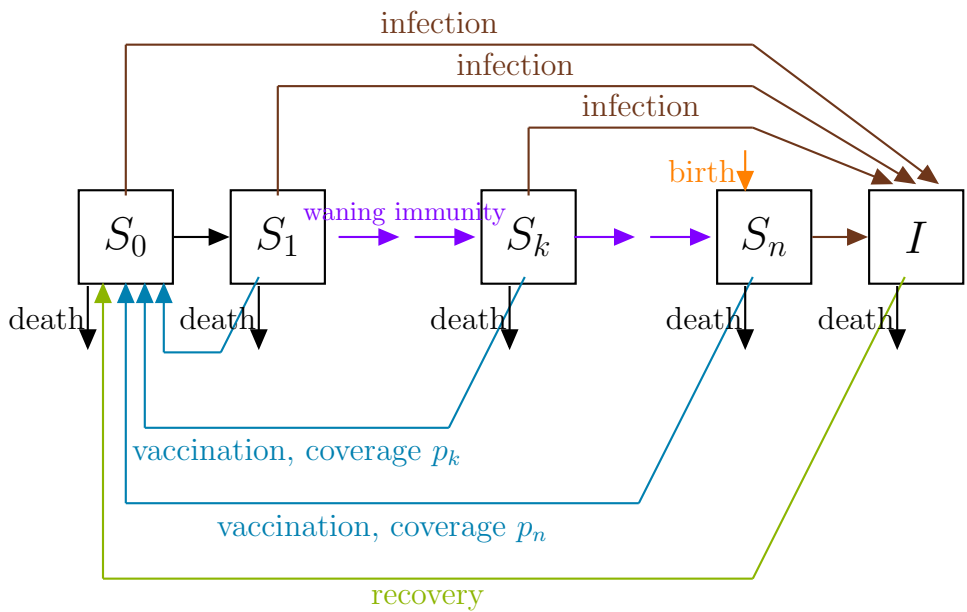
\begin{figure}[h!]
\centering
	\begin{tikzpicture}[auto, thick, node distance=2cm, >=triangle 45]
\draw
	node at  (0.6,-6) [sum, name=suma3] {\sumazero}
	node [block, right of=suma3] (inte2) {\sumaone};
\draw[->] (suma3) -- node {} (inte2);
\draw
	node at (5.9, -6) [sum, name=inte1] {\sumak}
	node at (9, -6) [sum, name=inte3] {\suman}
	node [block, right of=inte3] (Q2) {\sumai}
;
	\draw[->,auburn] (inte3) -- node {} (Q2);

		\draw[->,purple] (3.4,-6) -- node {} (4.2, -6);
		\draw[->,purple] (4.4,-6) -- node {\hspace{-34pt} \footnotesize waning immunity} (5.2, -6);
		\draw[->,purple] (6.5,-6) -- node {} (7.3, -6);
		\draw[->,purple] (7.5,-6) -- node {} (8.3, -6);
		\draw[-,applegreen] (10.85,-6.53) -- node {} (8.85, -10.53);
		\draw[-,applegreen](8.85, -10.53) -- node {recovery} (0.3, -10.53);
		\draw[->,applegreen](0.3, -10.53) -- node {} (0.3, -6.6);
		\draw[-,bluemunsell] (8.85,-6.53) -- node {} (7.35, -9.53);
		\draw[-,bluemunsell](7.35, -9.53) -- node {vaccination, coverage $p_n$} (0.6, -9.53);
		
		\draw[-,bluemunsell](4.85, -8.53) -- node {vaccination, coverage $p_k$} (0.85, -8.53);
		\draw[-,bluemunsell] (5.85,-6.53) -- node {} (4.85, -8.53);

		\draw[->,bluemunsell](0.6, -9.53) -- node {} (0.6, -6.6);
		
		\draw[->,bluemunsell](0.85, -8.53) -- node {} (0.85, -6.6);

		\draw[-,bluemunsell] (2.35,-6.53) -- node {} (1.85, -7.53);
		\draw[-,bluemunsell] (1.85, -7.53) -- node {} (1.1, -7.53);

		\draw[->,bluemunsell](1.1, -7.53) -- node {} (1.1, -6.6);
		
		\draw[-,auburn] (0.6,-5.45) -- node {} (0.6, -3.45);
		\draw[-,auburn](0.6, -3.45) -- node {infection} (8.85, -3.45);
		\draw[->,auburn](8.85, -3.45) -- node {} (11.3, -5.3);
		
		\draw[-,auburn] (2.6,-5.45) -- node {} (2.6, -4.0);
		\draw[-,auburn](2.6, -4.0) -- node {infection} (8.85, -4.0);
		\draw[->,auburn](8.85, -4.0) -- node {} (11, -5.3);
		
		\draw[-,auburn] (5.9,-5.45) -- node {} (5.9, -4.55);
		\draw[-,auburn](5.9, -4.55) -- node {infection} (8.85, -4.55);
		\draw[->,auburn](8.85, -4.55) -- node {} (10.7, -5.3);

		\draw[->,black](2.35, -6.65) -- node {\hspace{-34pt}death} (2.35, -7.50);
		\draw[->,black](0.1, -6.65) -- node {\hspace{-34pt}death} (0.1, -7.50);
		\draw[->,black](5.65, -6.65) -- node {\hspace{-34pt}death} (5.65, -7.50);
		\draw[->,black](8.75, -6.65) -- node {\hspace{-34pt}death} (8.75, -7.50);
		\draw[->,black](10.75, -6.65) -- node {\hspace{-34pt}death} (10.75, -7.50);
		\draw[->,orange](8.75, -4.85) -- node {\hspace{-34pt}birth} (8.75, -5.40);
\end{tikzpicture}
\caption{Symbolic representation of \eqref{eq:3c}}\label{fig}
\end{figure}

As introduced earlier, a central question in this paper has to do with determining the optimal choice for each $p_i$, resulting in a booster vaccination scheme that would not only minimize the system's $R_0$, driving the latter below $1$ for a realistic subset of parameter space and thus enabling local asymptotic stability for the DFE, but would also enable longer periods of time between successive immunizations. To address this question, in addition to studying the general properties of system \eqref{eq:3c}, we also compare two particular limiting cases of the above immunization scheme: In one, we set $p_n = 0$, while the remaining $p_i$ are allowed to vary (i.e., individuals in $S_n$ do not receive any vaccinations). In the other, individuals in the compartments $S_i, i < n$ are not immunized, so that $p_i = 0$ when $i < n$, and only $p_n$ is allowed to vary. 

In what follows, it will be convenient to introduce the following parameters:
\be\label{eq:param}
\omega_i:=p_i\omega,\quad  \delta_i:=(1-p_i)\delta.  
\ee

\paragraph{\bf Coverage of  all groups save $S_n$.}
Here, we allow for $\{p_i\}_{i=1}^{n-1}$ to be arbitrary but fix $p_n=0$ in \eqref{eq:3c}. This yields
\be\label{eq:3ca}
\begin{aligned}
&\frac{dS_0}{dt} =   \sum_{i=1}^{n-1}\omega_i
 S_i- \delta S_0 + rI - \beta_0 I S_0-\mu S_0; \\
&\frac{dS_i}{dt}  = -\omega_iS_i+\delta_{i-1}S_{i-1} - \delta_iS_i - \beta_i I S_i-\mu S_i, \hspace{3pt} i\neq \{0,n\}; \\
&\frac{dS_n}{dt} = \mu +\delta_{n-1} S_{n-1} - \beta_n I S_n-\mu S_n;\\
&\frac{dI}{dt} = I\sum_{i=0}^n \beta_i S_i  - rI-\mu I.
\end{aligned}
\ee
\paragraph{\bf Coverage of $S_n$.}
Here, we fix $p_i=0$ for $i<n$ but allow $p_n>0$. This yields
\be\label{eq:3aa}
\begin{aligned}
    \frac{dS_0}{dt} & = \omega_nS_n- \delta S_0 + rI - \beta_0 I S_0-\mu S_0; \\
    \frac{dS_i}{dt} &= \delta (S_{i-1} - S_i) - \beta_i I S_i-\mu S_i,\hspace{3pt} i\neq \{0,n\};\\
    \frac{dS_n}{dt} &= \mu-\omega_nS_n +\delta S_{n-1} - \beta_n I S_n-\mu S_n;\\
    \frac{dI}{dt} &=I\sum_{i=0}^n \beta_i S_i  - rI-\mu I.
\end{aligned}
\ee

We will  denote the disease-free equilibria values for $S_i, I$ by $S_{i, D}, I_{D}$ for all $0 \le i \le n$. As the disease is not present in the population at the DFE,  $I_{D} =0$. The endemic steady state solution for $S_i, I$ will be denoted by $S^{*}_{i}, I^{*}_{i}$. 
In addition, we will assume the normalization condition $\sum_{i}{S_i + I} = 1$, as in the standard density-dependent SIR model (we note that the $\sum_{i}{S_i + I}$ is conserved under the dynamics of \eqref{eq:3c}). We will further assume that $\beta_0<\beta_n$ (i.e., immunity wanes with time). By default, the vectors used will be row vectors. We will denote the standard basis for $\mathbb R^m$ by $\{{\bf e}_i\}$ (where $m$ will either be $n+1$ or $n+2$).

\section{Analytical Study of Models}\label{secIII}

 \subsection{Infection-Free Waning Immunity Dynamics}\label{sec2}

The first step of the analysis concerns the infection-free, no-vaccination dynamics obtained from \eqref{eq:3c} by setting $I(0)$ (and $p_i$ for all $i$) equal to 0. In this section, it is essentially verified that the model is capable of simulating the intended effect of waning immunity with reasonable accuracy (i.e., members of the $S_i$ compartment for each $i<n$ gradually drift into the $S_n$ compartment as time progresses). 

To this end, we prove the following result below: 

\begin{thm}\label{thm:inf_free}
For infection-free dynamics without vaccination (i.e., $I(0) = 0$, $p_i=0$ for all $i$) the system \eqref{eq:3c} has a unique equilibrium, given by $S_{i, D} = 0$ for $i\neq n$,  $S_{n, D} =1$. Moreover, it is globally exponentially stable, i.e., the transient solution reaches equilibrium exponentially fast for any initial state satisfying $I(0) = 0$.
\end{thm}
\vspace{-0.7cm}
\begin{proof}[Proof of Theorem \ref{thm:inf_free}]

The infection-free waning immunity dynamics are described by the solution of the initial value problem (IVP): 

\be\label{eq:2free}
\begin{aligned}
    &\frac{dS_0}{dt} = - \delta S_0 -\mu S_0; \\
    & \frac{dS_i}{dt} = \delta (S_{i-1} - S_i) -\mu S_i,\hspace{3pt}i\in\{1,\ldots,n-1\};\\
    & \frac{dS_n}{dt} = \mu +\delta S_{n-1}-\mu S_n;\\
    & S_i(0)=s_i,\hspace{3pt}  \sum_{i}s_i=1, s_i \ge 0,
\end{aligned}
\ee
where we note that, if $I(0) = 0$, by uniqueness of the solution for $I$, $I(t) = 0$ for all $t$. 

We note that the total population $S(t):=\sum_{i=0}^n S_i(t)$ is preserved under this dynamics, and that the non-negativity of the initial data $\{s_i\}$ ensures that $S_i(t)$ remains non-negative for all $t\in\mathbb{R}_+$. The solution of  \eqref{eq:2free} can be obtained explicitly in terms of  the associated Jacobian matrix ${\bf  J}$, 
 \[{\bf  J}= {\begin{pmatrix}\eta&0&0&\dots   &0&0\\\delta&\eta&0&\dots&0 &0  \\0&\delta&\eta&\dots &0&0  \\\vdots &\vdots&\vdots &\ddots&\vdots&\vdots  \\0 &0&0&\dots &\eta&0  \\0 &0 &0&\dots&\delta&-\mu \end{pmatrix}},\quad \eta:=-\delta-\mu,\]
using the variation of parameters method. The argument yielding the result is standard; we include it in Appendix \ref{sec:app} for completeness.
\end{proof}

\subsection{Disease-Free Equilibria: Location and Stability for Systems \eqref{eq:3c}, \eqref{eq:3ca}, and  \eqref{eq:3aa}}\label{sec3}

In this section, we prove the following result:

\vspace{-0.2cm}

\begin{thm}
\begin{enumerate}[label=(\alph*)]
\item[]
\item\label{it:th1} The DFE for \eqref{eq:3c}  is asymptotically stable if $\sum_{i=0}^n \beta_i S_{i, D} <r+\mu$ and unstable if $\sum_{i=0}^n \beta_i S_{i, D} >r+\mu$,  where $S_{i, D}$ are values uniquely determined by the configuration parameters and given explicitly in \eqref{eq:genS*}--\eqref{eq:genS*1}. 
We conclude that $R_0=\frac{\sum_{i=0}^n \beta_i S_{i, D}}{r+\mu}$   for \eqref{eq:3c}. 
\item\label{it:th2}
The DFE  and $R_0$ for \eqref{eq:3ca} are $w_i$-independent for all $i$. In fact, $R_0=\frac{\beta_n}{r+\mu}$ here.

\item\label{it:th3}
$R_0$ is strictly monotone decreasing in $\omega_n$ for \eqref{eq:3aa}. Moreover, for any  $\omega_n>0$, the value of $R_0$ for \eqref{eq:3aa} is strictly smaller than $\frac{\beta_n}{r+\mu}$, i.e., it is strictly smaller than the one for \eqref{eq:3ca}.

\item\label{it:th4} For sufficiently small $\delta$, the $R_0$ for \eqref{eq:3aa} becomes $\frac{\omega_{n} \beta_{0} + \mu \beta_{n}}{(\omega_{n} + \mu) (r + \mu)}$. In particular, the latter $R_0$ coincides with that of \eqref{eq:3ca} when $\omega_n=0$. 

\end{enumerate}
\end{thm}\label{thm:DFE_ka}

We begin by establishing values for the disease-free equilibria for both primary model equations.

\subsubsection{Location of Disease-Free Equilibria.}
We start with general facts concerning the DFE associated with  \eqref{eq:3c}. As $I_{D}= 0$, we obtain that the vector \[{\bf S}_{D}=(S_{0, D}, S_{1, D},\ldots, S_{n, D})\] must satisfy the linear matrix equation
\be\label{eq:matrforDFe}
{\bf A}{\bf S}_D^\intercal=-\mu {\bf e}_n^\intercal, \quad {\bf A}:= {\begin{pmatrix}d_0(0)&\omega_1&\dots &\omega_{k-1}& \omega_k & \dots &\omega_n\\\delta&d_1(0)&\dots&0 &0&\dots &0  \\0&\delta_1&\dots &0  &0& \dots &0 &  \\\vdots &\vdots &\ddots&\vdots&\vdots&\ddots &\vdots \\0 &0&\dots &\delta_{k-1}&d_k(0) & \dots &0  \\\vdots &\vdots&\ddots&\vdots&\vdots&\ddots &\vdots  \\0 &0&\dots&0&0& \dots&d_n(0)\end{pmatrix}},
\ee
where 
\be\label{eq:dexpress}
\begin{aligned} d_i(I)&:=-(\delta_i+\omega_i+\mu+\beta_iI)\mbox{ for } i\notin\{0,n\};\\ 
d_0(I)&:=-(\delta+\mu+\beta_0I);\ d_n(I):=-(\omega_n+\mu+\beta_n I).\end{aligned}
\ee
We first observe that ${\bf A}$ can be decomposed as ${\bf A}=\hat {\bf A} +{\bf e}_1^\intercal{\boldsymbol \omega}$, with the lower-triangular matrix $\hat {\bf A}$ and  ${\boldsymbol \omega}=(\omega_0,\ldots,\omega_n)$.
Using the matrix determinant lemma, we get
\[\det{{\bf A}}=(\det{\hat {\bf A}})\pa{1+{\boldsymbol \omega}\hat {\bf A}^{-1} {\bf e}_1^\intercal}.\]
The vector ${\bf u}^\intercal=\hat {\bf A}^{-1} {\bf e}_1^\intercal$, ${\bf u} = (u_1, u_2, \ldots, u_n, u_{n+1})$ can be evaluated explicitly using Cramer's rule:
\[
u_{k+1}= -\frac1{\abs{d_k(0)}}\prod_{i=0}^{k-1}\frac{\delta_i}{\abs{d_i(0)}}, 0 \le k \le n 
\]
with the convention that the empty product is equal to one.  This yields
\be\label{eq:detA}
\det{{\bf A}}=\pa{\prod_{i=0}^{n}{d_i(0)}} \pa{1-\sum_{k=0}^n\frac{\omega_k}{\abs{d_k(0)}}\prod_{i=0}^{k-1}\frac{\delta_i}{\abs{d_i(0)}}}.
\ee
In particular, ${\bf A}$ is always invertible.

We next note that the vector ${\bf v}=(1,\ldots,1)$ in $\mathbb R^{n + 1}$ is the left eigenvector for ${\bf A}$ with the eigenvalue $-\mu$. This implies the (expected) normalization condition 
\[\sum_{k=1}^n S_{k, D} ={\bf v} {\bf S}_{D}^\intercal=-\mu {\bf v}(A^{-1}{\bf e}_n^\intercal)={\bf v} {\bf e}_n^\intercal=1.\]

Using Cramer's rule,  the unique solution of  \eqref{eq:matrforDFe} is of the form
\be\label{eq:genS*}
S_{k, D} = \begin{cases} c\pa{\prod_{i=0}^{k-1}\delta_i}\pa{\prod_{i=k+1}^{n-1}\abs{d_i(0)}}& k<n\\  \frac\mu{\abs{d_n(0)}}+\frac c{\abs{d_n(0)}}\pa{\prod_{i=0}^{n-1}\delta_i} &k=n\end{cases}
\ee
for some constant $c$. This can be verified by directly substituting above expressions for $S_{k, D}$ into the system \eqref{eq:matrforDFe} and checking that all equations except the first one are satisfied one by one. The value of $c$  can be determined either from the first equation or by using the normalization condition $\sum_{k=1}^n S_{k, D}=1$. Since this sum is equal to $\frac\mu{\abs{d_n(0)}}\le 1$ for $c=0$ and is a monotonically increasing function of $c$, this value is unique and is non-negative. In fact, the explicit value of $c$ is given by 
\be\label{eq:genS*1}
c=\frac{\omega_n\mu}{\abs{\det{{\bf A}}}}.
\ee
When $\delta=0$, one gets
\be
\label{eq:EEdel}
\begin{aligned}
  & S_{0, D} =\frac{\omega_n }{\omega_n+\mu};\\
    & S_{i, D} =0,\quad i <n;\\
    &S_{n, D} =\frac{\mu }{\omega_n+\mu}.
\end{aligned}
\ee
In particular, for $\delta=0$, we have 
\be
\label{eq:expSa}
\sum_{i=0}^n\beta_i S_{i, D} =\frac{\omega_n \beta_0+\mu \beta_n }{\omega_n+\mu},
\ee
which will later be used to yield Theorem \ref{thm:DFE_ka}\ref{it:th4}.

We now consider the DFEs for specific values of the $p_i$s. The DFE for \eqref{eq:3ca} is obtained from \eqref{eq:genS*}--\eqref{eq:genS*1} by setting $\omega_n=0$, yielding
\be\label{DEFk} 
I_{D} = 0,\quad S_{i, D} =\begin{cases}0,& i<n,\\ 1,& i=n,
\end{cases}.
\ee
Here, the DFE does not depend on the value of the vaccination parameter $\omega$ and coverages $p_i$. As shown later,  $R_0$ for this system exhibits the same behavior.

To compute the DFE for \eqref{eq:3aa}, we note that in this case \eqref{eq:detA} gives 
\[\det{{\bf A}}=(-1)^{n+1}\nu(\omega_n+\mu)(\delta+\mu)^n, \quad \nu:=1-\pa{\frac{\delta }{\delta+\mu}}^{n}\frac{\omega_n}{\omega_n+\mu}.\] Substituting it into \eqref{eq:genS*}--\eqref{eq:genS*1} and setting $\delta_i=\delta$, $\omega_i=0$ for $i \le n-1$, we obtain

\be\label{DEFsn} S_{k, D} =\begin{cases}\frac1\nu\frac{\mu}{\delta+\mu}\frac{\omega_n}{\omega_n+\mu}\pa{\frac{\delta }{\delta+\mu}}^{k}&k<n\\ \frac1\nu\frac{ \mu }{\omega_n+\mu} & k=n\end{cases}.
\ee

In this case, we note that the DFE does depend on $\omega_n$, a fact that is occasionally stressed by writing $S_{i, D} (\omega_n)$ instead of $S_{i, D}$. In particular, when $\omega_n=0$ (no coverage), $S_{i, D}(0)$ coincides with solution \eqref{DEFk}. In the other extreme case, we have

\be\label{eq:expS}
\begin{aligned}
    & \lim_{\omega_n\to\infty} S_{k, D} (\omega_n)=\pa{\frac{\delta }{\delta+\mu}}^k\frac{\mu}{\hat\nu(\delta+\mu)} ,\quad k <n;\\
    & \lim_{\omega_n\to\infty} S_{n, D} (\omega_n) =0,
\end{aligned}
\ee
where $\hat \nu=1-\pa{\frac{\delta }{\delta+\mu}}^{n}$.

\subsubsection{Linear Stability Analysis for Disease-Free Equilibria.}
Here we continue to prove Theorem \ref{thm:DFE_ka}.

The Jacobian  ${\bf  J}$ for \eqref{eq:3c} is given by 
 \be\label{eq:2aa} 
 {\bf  J}=
{\begin{pmatrix}d_0(I)&\omega_1&\dots   &\omega_{k-1}&\omega_k&\dots &\omega_{n-1}&\omega_{n}&r-\beta_0 S_o\\\delta&d_1(I)&\dots&0&0 &\dots &0 &0&-\beta_1 S_1 \\0&\delta_1&\dots &0&0 &\dots&0 &0 &-\beta_2 S_2  \\\vdots &\vdots &\ddots&\vdots&\vdots&\ddots&\vdots &\vdots&\vdots  \\0 &0&\dots & \delta_{k-1}&d_{k}(I)&\dots&0&0 &-\beta_{k} S_{k} \\\vdots &\vdots &\ddots&\vdots&\vdots&\ddots &\vdots  &\vdots &\vdots  \\0 &0 &\dots&0&0&\dots&d_{n-1}(I)&0 &-\beta_{n-1} S_{n-1}\\0 &0 &\dots&0&0&\dots&\delta_{n-1}&d_n(I) &-\beta_n S_n\\\beta_0I&\beta_1I&\dots   &\beta_{k-1}I&\beta_kI&\dots&\beta_{n-1}I&\beta_{n}I&-r-\mu+B \end{pmatrix}},\ee
with $B:=\sum_{i=0}^n \beta_i S_i$. For the DFE, this matrix is reduced to 
 \[{\bf J}= {\begin{pmatrix}d_0(0)&\omega_1&\dots &\omega_{k-1}& \omega_k & \dots &\omega_n&r-\beta_0 S_{0, D} \\\delta&d_1(0)&\dots&0 &0&\dots &0 &-\beta_1 S_{1, D} \\0&\delta_1&\dots &0  &0& \dots &0 & -\beta_2 S_{2, D} \\\vdots &\vdots &\ddots&\vdots&\vdots&\ddots &\vdots&\vdots  \\0 &0&\dots &\delta_{k-1}&d_k(0) & \dots &0 &-\beta_k S_{k, D} \\\vdots &\vdots&\ddots&\vdots&\vdots&\ddots &\vdots & \vdots \\0 &0&\dots&0&0& \dots&d_n(0)&-\beta_n S_{n, D} \\0&0&\dots &0  &0&\dots &0&B_D -r-\mu \end{pmatrix}},\]
  where $B_D:=\sum_{i=0}^n \beta_i S_{i, D}$.

We now perform Laplace expansion of the determinant by the last row to evaluate the characteristic polynomial $p(z)=\det\pa{{\bf J}-z\mathds{1}}$ of ${\bf J}$. This  yields 
\be \label{eq:p(z)}
p(z)=\pa{-r-\mu+\sum_{i=0}^n \beta_i S^*_i-z} \det({\bf M}-zI), 
\ee
where 
\be\label{eq:Gmat}
{\bf M}=\begin{pmatrix}d_0(0)&\omega_1&\omega_2&\dots &\omega_{k-1}& \omega_k & \dots &\omega_n\\\delta&d_1(0)&0&\dots&0 &0&\dots &0 \\0&\delta_1&d_2(0)&\dots &0  &0& \dots &0 \\\vdots &\vdots&\vdots &\ddots&\vdots&\vdots&\ddots &\vdots \\0 &0&0&\dots &d_{k-1}(0)&0 & \dots&0 \\0 &0&0&\dots &\delta_{k-1}&d_k(0) & \dots&0 \\\vdots &\vdots&\vdots &\ddots&\vdots&\vdots&\ddots &\vdots \\0 &0&0&\dots&0&0& \dots&d_n(0)\end{pmatrix}.
\ee
In particular, by the Hartman-Grobman theorem (see e.g., \cite[Theorem 9.9]{T}), the system is asymptotically stable  if 
\begin{enumerate}
\item[]
\begin{enumerate}
\item $-r-\mu+\sum_{i=0}^n \beta_i S_{i, D} <0$.
\item $\sigma({\bf  M})\subset \mathbb P_L$, where $\mathbb P_L$ is the left  halfplane, 
$\mathbb P_L = {(-\infty,0)} \times \mathbb R$.
\end{enumerate} 
\end{enumerate} 
This is consistent with the next-generation matrix approach, e.g., \cite[Theorem 2]{DW}, where the matrices $F = \sum_{i=0}^n \beta_i S_{i, D}$ and $V = \mu + r$ in our context, so that the next-generation matrix is equal to $(r+\mu)^{-1}\sum_{i=0}^n \beta_i S_{i, D}$. Here \cite[Assumption (A5)]{DW} corresponds to the condition (b) above.

To show that $\sigma({\bf M})\subset \mathbb P_L$, we use the Ger{\v{s}}gorin circle theorem: Letting $D_r(a)$ denote a closed disc in the complex plane centered at $a\in\mathbb C$ with radius $r$, the Ger{\v{s}}gorin discs for an $(n+1)\times (n+1)$ matrix ${\bf M}$ are defined by  $\{D_{r_i}({\bf M}_{ii})\}_{i=1}^{n+1}$ with the radii $r_i=\sum_{j\neq i}\abs{{\bf M}_{ij}}$.  The Ger{\v{s}}gorin circle theorem, see \cite{HJ}, stipulates that the spectrum $\sigma({\bf M})$ of ${\bf M}$ is located in the union of these discs, $\sigma({\bf M})\subset \cup_{i=1}^{n+1}D_{r_i}({\bf M}_{ii})$.

For the matrix ${\bf  M}$ in \eqref{eq:Gmat}, we have ${\bf M}_{ii}=-(\omega_i+\delta_i+\mu)$, while $r_i=\omega_i+\delta_i$ (with convention $\omega_0=\delta_n=0$). In particular, $\cup_{i=1}^{n+1}D_{r_i}({\bf M}_{ii})$ is contained in the left half plane $\{z\in \C: \ Re(z)\le-\mu\}$. Hence the Ger{\v{s}}gorin circle theorem implies that all but one of the eigenvalues of $V$ have real part $\le-\mu$. On the other hand, the remaining eigenvalue of ${\bf J}$ is $-r-\mu+\sum_{i=0}^{n+1} \beta_i S_{i, D}$ (as can be seen from the expression for $p(z)$ in \eqref{eq:p(z)}).

We see that Theorem \ref{thm:DFE_ka}\ref{it:th1} holds for \eqref{eq:3c} and its cases by the above analysis. Substituting the value of the DFE for model \eqref{eq:3ca} in \eqref{DEFk}, we obtain that $-r - \mu + \sum_{i=0}^{n+1} \beta_i S_{i, D} = - r - \mu + \beta_n$, and Theorem \ref{thm:DFE_ka}\ref{it:th2} follows.

We claim that the threshold for the stability of the DFE for  \eqref{eq:3aa}, namely $T:=\sum_{i=0}^n \beta_i S_{i, D}$, monotonically decreases as $\omega_n \ge 0$ increases. Indeed, expressing $T$ as a function of $\xi:=\frac{\omega_n}{\omega_n+\mu}$, denoting $\sigma:=\frac{\delta}{{\delta+\mu}}$, and using $\omega_n=\frac{\mu \xi}{1-\xi}$, we compute
\[T=\frac{A\xi}{1-\sigma^n \xi}+\frac{\beta_n\pa{1-\xi}}{1-\sigma^n \xi},\]
with
\[A=\frac\mu{\pa{\delta+\mu}}\sum_{i=0}^{n-1}\beta_i\sigma^i.\]

Rearranging, we get
\be\label{eq:T}T=T_o+\frac{\pa{A-\beta_n}\sigma^{-n}+\beta_n}{1-\sigma^n \xi},\ee
where $T_o$ is $\xi$-independent. Using $\beta_0 \le \beta_1 \le \ldots \le \beta_n$, we can estimate
\[\pa{A-\beta_n}\sigma^{-n}+\beta_n\le \pa{\frac\mu{{\delta+\mu}}\sum_{i=0}^{n-1}\pa{\sigma^i}-1}\sigma^{-n}\beta_n+\beta_n=0,\]
so the numerator in \eqref{eq:T} is non-positive (and is in fact strictly negative if $\beta_0<\beta_n$). When $\omega_n \ge 0$, $\xi$ is monotone increasing in $\omega_n$, so we deduce that $1-\sigma^n \xi > 0$ decreases as $\omega_n$ increases, which in turn implies that $T$ is monotone decreasing in $\omega_n$ for a nonnegative $\omega_n$, as claimed. 

We note that if $\beta_n<r+\mu$, then we have that
\[-r-\mu+\sum_{i=0}^n \beta_i S_{i, D} < -r-\mu+\beta_n\sum_{i=0}^n  S_{i, D} =\beta_n-r-\mu<0,\] where we have used the fact that $\beta_0<\ldots<\beta_n$. Thus the DFE for \eqref{eq:3aa} is asymptotically stable if $\beta_n<r+\mu$, and Theorem \ref{thm:DFE_ka}\ref{it:th3} follows. 

Finally,  Theorem \ref{thm:DFE_ka}\ref{it:th4} has been already established earlier, see \eqref{eq:expSa} above.
\qed

\begin{rem}
When $\omega\to\infty$, we have an asymptotically stable DFE for \eqref{eq:3aa} provided that \[\sum_{i=0}^{n-1}\beta_i\pa{\frac{\delta }{\delta+\mu}}^i\frac{\mu}{\nu(\delta+\mu)}<r+\mu,\]
where $\nu$ is as defined above.  
For $\delta \ll 1$, this reduces to the condition $\beta_0<r+\mu$ (cf. with  $\beta_{n}<r+\mu$ for \eqref{eq:3ca}). 

\end{rem}

\subsection{Endemic Equilibria: Location, Existence, and Uniqueness for \eqref{eq:3ca} and  \eqref{eq:3aa}}\label{sec4}
We begin by considering the location and uniqueness of the endemic equilibria for each model system. To this end, we will prove the following  result:

\vspace{-.2cm}

\begin{thm}\label{thm:1}
There exists $\delta_o>0$ such that for $\delta\le\delta_o$, if \[\pa{\omega_n+\mu}\pa{\mu+r}>{\beta_0\omega_n+\beta_n\mu},\] then \eqref{eq:3c} has no realistic endemic equilibrium and  if \[\pa{\omega_n+\mu}\pa{\mu+r}<{\beta_0\omega_n+\beta_n\mu},\] then \eqref{eq:3c} has exactly one endemic equilibrium.  The value for $I^*$ has to lie in the intervals $ J_{1,2}$ described in Lemma \ref{lem:1} below. 
\end{thm}

\begin{rem}
We note that for system \eqref{eq:3ca} $\omega_n=0$, in which case the conditions above become $r+\mu>\beta_n$ and $r+\mu<\beta_n$, respectively. \end{rem}

\textit{Proof of Theorem \ref{thm:1}.}
It is first noted that
\be\label{eq:endco}
\begin{aligned}
& I^*+\sum_{i=0}^nS_i^*= 1;\\
& \sum_{i=0}^n \beta_i S^*_i =r+\mu
\end{aligned}
\ee 
for any equilibrium with $I^*\neq0$ (where the conservation law  $I(t)+\sum_{i=0}^nS_i(t)\equiv 1$ is used in the first equality).

The endemic equilibrium for \eqref{eq:3c} is given by 
\be\label{eq:3aae}
\begin{aligned}
& 0 = \sum_{i=1}^{n}\omega_i S_i^*+ d_0(I^*) S_0^* + rI^* ; \\
& 0= \delta_{i-1}S_{i-1}^* + d_i(I^*) S_i^*,\quad i\neq \{0,n\}; \\ 
& 0 = \mu+\delta_{n-1}S_{n-1}^* + d_n(I^*) S_n^*,
\end{aligned}
\ee
where $d_i$ are defined in \eqref{eq:dexpress}. We next express the vector ${\bf S}^*:=(S_0^*,\ldots,S_n^*)$ in terms of $I^*$ by writing 
as ${\bf R}^\intercal= {\bf A}_\delta({\bf S}^*)^\intercal$, where ${\bf R}=-(rI^*,0,\ldots,0,\mu)$ and the $(n+1)\times (n+1) $ matrix ${\bf A}_\delta$ is given by
 \be\label{eq:2p}
 {\bf A}_\delta=  \begin{pmatrix}d_0(I^*)&\omega_1&\dots   &\omega_{k-1}&\omega_k&\dots &\omega_{n-1}&\omega_{n}\\\delta&d_1(I^*)&\dots&0&0 &\dots &0 &0 \\0&\delta_1&\dots &0&0 &\dots&0 &0   \\\vdots &\vdots &\ddots&\vdots&\vdots&\ddots&\vdots & \vdots   \\0 &0&\dots & \delta_{k-1}&d_{k}(I^*)&\dots&0&0  \\\vdots &\vdots &\ddots&\vdots&\vdots&\ddots &\vdots  &\vdots  \\0 &0 &\dots&0&0&\dots&d_{n-1}(I^*)&0 \\0 &0 &\dots&0&0&\dots&\delta_{n-1}&d_n(I^*)  \end{pmatrix}.
 \ee
 
 In particular, if ${\bf A}_\delta$ is invertible (which  will be shown to be true for small values of $\delta$), we get $({\bf S}^*)^\intercal={\bf  A}_\delta^{-1}{\bf R}^\intercal$.
Let ${\boldsymbol \beta}=(\beta_0,\ldots,\beta_n)$, then  \eqref{eq:endco} yields 
\be\label{eq:Fdef}
r+\mu={\boldsymbol \beta}\cdot {\bf S}^*={\boldsymbol \beta} \pa{{\bf A}_\delta^{-1}{\bf R}^\intercal}=:F_\delta(I^*).
\ee

We rewrite the latter equation as 
\be\label{eq:Feq}
r+\mu-F_0(I^*)=F_\delta(I^*)-F_0(I^*).
\ee
The idea now is to show that 
\be\label{eq:normb}
\abs{F_\delta(I^*)-F_0(I^*)}\le K\delta
\ee
for some $\delta$ independent constant $K$, for sufficiently small values of $\delta$. This would imply that $I^*$ has to satisfy the inequality
\be\label{eq:normba}
\abs{r+\mu-F_0(I^*)}\le K\delta,
\ee
which will allow us to determine the location of $I^*$ within $\sqrt\delta$ precision using P{\'o}lya's theorem on polynomials. We then will show the uniqueness of the corresponding solution.

We proceed by verifying the details of these steps. To get \eqref{eq:normb} the following estimates will be used (see, e.g.,  \cite[Lemma 7.18]{Dym}):
\be\label{eq:resoldiff}
\begin{aligned}
\norm{{\bf A}_\delta^{-1}}&\le\frac{\norm{{\bf A}_0^{-1}}}{1-\norm{{\bf A}_\delta-{\bf A}_0}\norm{{\bf A}_0^{-1}}};\\ \norm{{\bf A}_\delta^{-1}-{\bf A}_0^{-1}}&\le \norm{{\bf A}_\delta-{\bf A}_0}\frac{\norm{{\bf A}_0^{-1}}^2}{1-\norm{{\bf A}_\delta-{\bf A}_0}\norm{{\bf A}_0^{-1}}},
\end{aligned}
\ee
provided $\norm{{\bf A}_\delta-{\bf A}_0}\norm{{\bf A}_0^{-1}}<1$. In this case,

\be\label{eq:Adiff1}
{\bf A}_\delta-{\bf A}_0={\begin{pmatrix}-\delta & 0 & 0 & \dots & 0 & 0 \\ \delta & -\delta_1 & 0  & \dots & 0 & 0 \\ 0 & \delta_1 & -\delta_2 & \ldots & 0 & 0 \\\vdots &\vdots &\vdots&\ddots &\vdots & \vdots \\0 & 0 &0 & \ldots & -\delta_{n-1} & 0 \\ 0 & 0 & 0 & \ldots & \delta_{n-1} & 0 \end{pmatrix}},
\ee

and thus 
\be\label{eq:Adiff}
 \norm{{\bf A}_\delta- {\bf A}_0}\le 2\delta,
\ee
where we have used the Schur test, \cite{S},
\be\label{eq:Holmgren}
\|{\bf M}\|\leq {\sqrt {\|{\bf M}\|_{1}\|{\bf M}\|_{\infty }}}.
\ee
Here $\|{\bf M}\|_{1}$ is the matrix $1$-norm of an $n\times n$ matrix ${\bf M}$, given by
\[\|{\bf M}\|_{1}=\max_j\sum_{i=1}^n\abs{{\bf M}_{ij}},\]
and $\|{\bf M}\|_{\infty}$ is  the matrix $\infty$-norm of an $n\times n$ matrix ${\bf M}$, given by
\[\|{\bf M}\|_{\infty}=\max_i\sum_{j=1}^n\abs{{\bf M}_{ij}},\]
see, e.g., \cite{HJ}.

On the other hand, ${\bf A}_0^{-1}$ can be evaluated explicitly:

\be\label{eq:A_0}
{\bf A}_0^{-1}={\begin{pmatrix}1/\hat d_0(I^*)& - \hat \omega_1  & \dots& - \hat \omega_n  \\ 0 & 1/\hat d_1(I^*) & \dots  & 0    \\\vdots &\vdots &\ddots & \vdots \\0 &0  & \dots &  1/\hat d_n(I^*) \end{pmatrix}},
\ee
where each $\hat d_i(I^*)$ is obtained from $d_i(I^*)$ by setting $\delta=0$ there, and 
\[\hat \omega_i = \frac{\omega_i}{\hat d_{0}(I^{*}) \hat d_{i}(I^{*})}.\]
Using \eqref{eq:Holmgren} and $\beta_0\le \beta_i$ for all $i$, we can bound
\be\label{eq:A_0norm}
\norm{{\bf A}_0^{-1}}\le \frac{\sqrt{n+1}}{\beta_0I^*+\mu}.
\ee
Hence  
\be\label{eq:difnapr}
\norm{{\bf A}_\delta- {\bf A}_0}\norm{{\bf A}_0^{-1}}<1/2
\ee 
for $\delta$ sufficiently small.

Combining \eqref{eq:resoldiff}, \eqref{eq:Adiff}, \eqref{eq:A_0norm}, and \eqref{eq:difnapr}, we get 
\be\label{eq:normdif}
\abs{F_\delta(I^*)-F_0(I^*)}\le \norm{{\boldsymbol \beta}}\norm{ {\bf A}_\delta^{-1}- {\bf A}_0^{-1}}\norm{{\bf R}}\le \frac{4 (n+1)^{3/2}\beta_n\pa{r+\mu}}{\pa{\beta_0I^*+\mu}^2}\delta, \ee
and \eqref{eq:normb} follows.

We now state the implications of the bound \eqref{eq:normba}. 

\vspace{-.2cm}

\begin{lem}\label{lem:1}
There exists a constant $C$ such that for $\delta$ sufficiently small, a set of possible values $I^*$ for endemic equilibria in \eqref{eq:3c} is contained in the union of two intervals 
\be\label{eq:roots}
J_1=\left[y_1- C\sqrt{\delta} ,y_1+C\sqrt{\delta}\right],\quad J_2=\left[y_2 - C\sqrt{\delta},y_2+C\sqrt{\delta}\right]
\ee
where  $y_{1,2}$ are roots of the equation
\be\label{eq:rootsa}
1+\frac{r}{\beta_0x+\mu}-\frac{\beta_n}{\beta_n x+\mu+\omega_n}-\frac{\omega_n\beta_0}{\pa{\beta_0x+\mu}\pa{\beta_n x+\mu+\omega_n}}=0.
\ee
In particular, if this equation has no real roots or no roots in $[0,1]$, then for $\delta$ small enough we do not have endemic equilibria. This is the case for $\pa{\omega_n+\mu}\pa{\mu+r}>{\beta_0\omega_n+\beta_n\mu}$. If $\pa{\omega_n+\mu}\pa{\mu+r}<{\beta_0\omega_n+\beta_n\mu}$, then there is exactly one root $y$ in $[0,1]$ (with $y=0$ for $\pa{\omega_n+\mu}\pa{\mu+r}={\beta_0\omega_n+\beta_n\mu}$). 
\end{lem}

{\it Proof}: 
Note that $F_0$ can be evaluated explicitly, namely

\be\label{eq:F0}
\begin{aligned}
F_0(I^*)&={\boldsymbol \beta} \pa{{\bf A}_0^{-1}{\bf R}^\intercal}=\beta_0\pa{\mu \hat \omega_n-\frac{rI^*}{\hat d_0(I^*)}}-\beta_n\frac{\mu}{\hat d_n(I^*)}
\\&=\frac{\beta_0\mu\omega_n}{\pa{\beta_0I^*+\mu}\pa{\beta_n I^*+\mu+\omega_n}}+\frac{\beta_0rI^*}{\beta_0I^*+\mu}+\frac{\beta_n\mu}{\beta_n I^*+\mu+\omega_n},
\end{aligned}
\ee
using \eqref{eq:A_0}. Thus, after simplifications, we see that $I^*$ is a solution of the equation
\be
r+\mu-F_0(x)=\mu+\frac{r\mu}{\beta_0x+\mu}-\frac{\mu\beta_n}{\beta_n x+\mu+\omega_n}-\frac{\omega_n\mu\beta_0}{\pa{\beta_0x+\mu}\pa{\beta_n x+\mu+\omega_n}}.
\ee

We now proceed to consider the equation given $\beta_0 > 0$, and will return to the case $\beta_0 = 0$ in Remark \ref{b0=0} below. 
 
In this case, taking the common denominator yields
\be\label{eq:cden}
r+\mu-F_0(x)=Q(x)\frac{\mu\beta_0\beta_n}{\pa{\beta_0x+\mu}\pa{\beta_n x+\mu+\omega_n}},\quad Q(x)=x^2+ax+b,
\ee
with
\be\label{eq:ab}a=\frac{\beta_0\pa{\mu+\omega_n}+\beta_n\pa{\mu+r-\beta_0}}{\beta_0\beta_n},\quad b=\frac{\pa{\omega_n+\mu}\pa{\mu+r}-\pa{\beta_0\omega_n+\beta_n\mu}}{\beta_0\beta_n}.\ee
Hence  \eqref{eq:normba} can be rewritten as 
\be\label{eq:normbar}
\abs{Q(I^*)}\le \tilde C\delta \pa{\beta_0I^*+\mu}\pa{\beta_n I^*+\mu+\omega_n},
\ee
and since $I^*\in[0,1]$, we deduce that if $I^*$ satisfies \eqref{eq:normbar}, it must also satisfy
\be\label{eq:normbarm}
\abs{Q(I^*)}\le \tilde C\delta \max_{x\in[0,1]}\pa{\beta_0x+\mu}\pa{\beta_n x+\mu+\omega_n}=\hat C\delta.
\ee
 This means that, if $\mathcal B$ denotes the set 
\be\label{eq:setB}\mathcal B=\{x\in[0,1]:\ \abs{Q(x)}\le\hat C\delta\},\ee
 then any solution $I^*$ of \eqref{eq:3aae} is contained in $\mathcal B$.

Since $Q(x)$ is a continuous function, we know that $\mathcal B$ is the union of disjoint closed interval(s), and that $Q(x)$ assumes the values $\pm\hat C\delta$ at the endpoints of these intervals. There are at most two intervals $\tilde J_{1,2}$ such that $ \mathcal B=\tilde J_1\cup \tilde J_2$, since a quadratic polynomial $Q$ can assume any value at most twice.

The size (namely, the Lebesgue measure) of the set $\mathcal B$ now needs to be estimated.
To this end, we use P{\'o}lya's theorem, see \cite{P}: For  a monic polynomial  $P$ of degree $q$, 
\be\label{Polya} m({\{x\in \mathbb {R} :\abs{P(x)}\leq \epsilon\}})\leq 4\left({\frac {\epsilon}{2}}\right)^{1/q},\quad \epsilon>0,
\ee
where $m\pa{ \cdot }$ stands for the Lebesgue measure of a set.

This enables us to bound a measure of the set $\mathcal B$ defined in \eqref{eq:setB}, namely ($q = 2$ and with the choice  $\epsilon=\hat C\delta$)
\be
m\pa{\mathcal B}\le 2\sqrt{2\hat C\delta}.
\ee
If $Q(x)$ does not have real roots then $\mathcal B=\emptyset$  for $\delta$ sufficiently small, so there cannot be an endemic equilibrium. If the roots $y_{1,2}$ of $Q(x)$ - that is, the roots of \eqref{eq:rootsa} are real,
\be\label{roots}
y_{i}:\ Q(y_i)=0,\quad y_1\le y_2,
\ee
then they are contained in $\mathcal B$, where the latter consists of at most two intervals. Moreover,  we deduce from symmetry considerations that in this case $\mathcal B\subset J_1\cup J_2$, where 
\be\label{eq:strucB}
J_1=\left[y_1- \sqrt{2\delta \hat C} ,y_1+\sqrt{2\delta \hat C}\right],\quad J_2=\left[y_2- \sqrt{2\delta \hat C} ,y_2+\sqrt{2\delta \hat C}\right].
\ee
This implies that the solutions of \eqref{eq:3aae} for $I^*\in(0,1]$ must satisfy
\be\label{eq:strucI}
I^*\in \pa{J_1\cup J_2}\cap [0,1].
\ee
Hence if $y_1<0$, we see that $I^*\in J_2$, and if $y_1>1$, we  infer that there are no endemic equilibria for sufficiently small $\delta$. This establishes \eqref{eq:roots}.

We now analyze the conditions under which \eqref{eq:rootsa} has roots in $[0,1]$. To begin with, evaluating the left hand side of \eqref{eq:rootsa} for $x=0$ yields \[\frac{\pa{\omega_n+\mu}\pa{\mu+r}-\pa{\beta_0\omega_n+\beta_n\mu}}{\mu\pa{\mu+\omega_n}},\] so the left-hand side (lhs) is negative for $\pa{\omega_n+\mu}\pa{\mu+r}<{\beta_0\omega_n+\beta_n\mu}$ at this value for $x$. On the other hand, evaluating the left hand side  for $x=1$ gives
\[\frac{\mu\pa{\mu+\omega_n+\beta_0}+r\pa{\mu+\omega_n+\beta_n}}{\pa{\beta_0+\mu}\pa{\beta_n+\mu+\omega_n}}>0.\] This implies by the intermediate value theorem that there must be a root $y$ for \eqref{eq:rootsa}  if \[\pa{\omega_n+\mu}\pa{\mu+r}<{\beta_0\omega_n+\beta_n\mu}.\] Moreover, upon taking the common denominator in  \eqref{eq:rootsa} we get a quadratic polynomial in the numerator (a multiple of $Q$ with a positive leading coefficient), which is negative at $0$ and is positive at $1$. This in fact implies that there is only one root $y_2$ in $(0,1)$, with the second root being strictly negative in this case. 

If $\pa{\omega_n+\mu}\pa{\mu+r}>{\beta_0\omega_n+\beta_n\mu}$, the coefficients $a,b$ in \eqref{eq:ab} are positive. For $b$, this follows directly from \eqref{eq:ab}, while for $a$ it follows from the observation that $\pa{\omega_n+\mu}\pa{\mu+r}>{\beta_0\omega_n+\beta_n\mu}$ is equivalent to
\[\frac{\beta_0\omega_n}{\omega_n+\mu}+\frac{\beta_n\mu}{\omega_n+\mu}<\mu+r,\]
and from the inequality
\[\beta_0<\frac{\beta_0\omega_n}{\omega_n+\mu}+\frac{\beta_n\mu}{\omega_n+\mu}.\]
These relations imply that $\beta_0<\mu+r$ and the positivity of $a$ is now straightforward.

Since $a,b$ are both positive, we deduce that $Q$  (and thus \eqref{eq:rootsa} as well) has no roots in $[0,1]$ in this case. \qed

\begin{rem}\label{b0=0}

If $\beta_0 = 0$, \eqref{eq:cden} becomes \[r+\mu-F_0(x)= \tilde Q(x) \frac{1}{\pa{\beta_n x+\mu+\omega_n}},\] where $\tilde Q(x) = \beta_n(r + \mu)x + (r + \mu)(\omega_n + \mu) - \mu \beta_n$. If $(r + \mu)(\omega_n + \mu) > \beta_n \mu$, $\tilde Q(x)$ has no roots in $[0, 1]$. If $(r + \mu)(\omega_n + \mu) < \beta_n \mu$, then since $Q(0) < 0$ and $Q(1) = \beta_n(r + \mu) + (r + \mu)(\omega_n + \mu) - \mu \beta_n > 0$, $Q(x)$ has a unique root in $[0, 1]$. Hence in this case Theorem \ref{thm:1} follows immediately (with $\beta_0 = 0$ in the condition there).

\end{rem}
\vspace{.2cm}

The next assertion shows the uniqueness of the corresponding solution for $\beta_0 \neq 0$.

\vspace{-.2cm}

\begin{lem}\label{lem:2}
Let $y_{1,2}$ and $C$ be as in Lemma \ref{lem:1}, and let $J_{1,2}$ be the intervals in that lemma.  Assume that $ J_i\subset(0,1]$ for some (or both) values of $i=1,2$. Then, for $\delta$ sufficiently small there exists a unique endemic equilibrium for \eqref{eq:3c} with $I^*\in J_i$, provided $\abs{y_1-y_2}\ge\delta^{1/3}$.
\end{lem}
This completes the proof of Theorem \ref{thm:1}.
\qed
\begin{proof}[Proof of Lemma \ref{lem:2}]
The starting point here is \eqref{eq:Feq}.  Taking the common denominator on the right hand side (cf. \eqref{eq:cden}) and using the fact that the $y_{1,2}$ used in the proof of the previous lemma are the roots of the resulting numerator yields the identity
\be\label{eq:EE2}
\frac{ \beta_0\beta_n\mu\pa{I^*-y_1}\pa{I^*-y_2} }{\pa{\beta_0I^*+\mu}\pa{\beta_n I^*+\mu+\omega_n}}=F_\delta(I^*)-F_0(I^*).
\ee
Suppose that the interval $\tilde J_1$ satisfies the assumptions of the lemma (the argument for $\tilde J_2$ will be identical). The above equation is then brought into the form
\be\label{eq:tildeG}
I^*-y_1=\tilde G(I^*),
\ee
where
\be
\tilde G(x)=\frac{{\pa{\beta_0x+\mu}\pa{\beta_n x+\mu+\omega_n}}}{ \beta_0\beta_n\mu\pa{x-y_2} }\pa{F_\delta(x)-F_0(x)}.
\ee
 Let $u=x-y_1$ and let $G$ be a function given by $G(u)=\tilde G(u+y_1)$.  \eqref{eq:tildeG} then reduces to
\be\label{eq:G}
u= G(u).
\ee
We claim that 
\begin{enumerate}[label=(\alph*)]
\item\label{it:a} The interval $S:=\tilde J_1-y_1$ is invariant under $G$, i.e., $G(S)\subset S$; 
\item\label{it:b}  $G$ is a contraction on $S$, i.e., $\abs{G(u)-G(v)}\le \alpha\abs{u-v}$ for $u,v\in S$ and $\alpha<1$.
\end{enumerate}
The items \ref{it:a}  and \ref{it:b} imply that the Banach fixed-point theorem is applicable in this context and so the solution $u^*$ of \eqref{eq:G} exists and is unique in $S$, and the original assertion follows. To establish \ref{it:a}, we note that, since  $\abs{y_1-y_2}\ge \delta^{1/3}$ by the above assumption, for $x\in \tilde J$ we have, by \eqref{eq:normdif} with some $D>0$, 

\be
\abs{\tilde G(x)}\le \max_{x\in[0,1]}\frac{{\pa{\beta_0x+\mu}\pa{\beta_n x+\mu+\omega_n}}}{ \beta_0\beta_n\mu\delta^{1/3} }\,\max_{x\in[0,1]}\abs{F_\delta(x)-F_0(x)}\le D\delta^{2/3}
\ee
for $\delta$ sufficiently small. Since
\[S=\left[- C\sqrt{\delta} ,C\sqrt{\delta}\right]\supset\left[-D\delta^{2/3},D\delta^{2/3}\right],\]
we see that \ref{it:a}  holds for such $\delta$. 

To check \ref{it:b}, the derivative $G'$ on $S$ is bounded using the product rule. We first compute (see \eqref{eq:Fdef})
\be\label{eq:prrulF}
\partial_x({F_\delta(x)-F_0(x)})={\boldsymbol \beta} \,\partial_x({{\bf A}_\delta^{-1}- {\bf A}_0^{-1}}){\bf R}^\intercal+{\boldsymbol \beta}({{\bf A}_\delta^{-1}-{\bf A}_0^{-1}})\partial_x{\bf R}^\intercal.
\ee
Next, we note that $\partial_x{\bf R}^\intercal=r{\bf e}_1^\intercal$, and since we have the resolvent identity

\[{\bf A}_\delta^{-1}-{\bf A}_0^{-1}={\bf A}_\delta^{-1}\pa{{\bf A}_0-{\bf A}_\delta}{\bf A}_0^{-1},\]
and ${\bf A}_0-{\bf A}_\delta$ is $x$-independent (see \eqref{eq:Adiff1}), we get that
\be\label{eq:prrul}
 \partial_x({{\bf A}_\delta^{-1}-{\bf A}_0^{-1}})=({ \partial_x {\bf A}_\delta^{-1}})\pa{{\bf A}_0- {\bf A}_\delta}{\bf A}_0^{-1}+{\bf A}_\delta^{-1}\pa{{\bf A}_0-{\bf A}_\delta}({\partial_x {\bf A}_0^{-1}}).
\ee
Bounding the right-hand side uses the relations
\[{ \partial_x {\bf A}_\delta^{-1}}=-{\bf A}_\delta^{-1}\pa{ \partial_x {\bf A}_\delta} {\bf A}_\delta^{-1},\quad { \partial_x {\bf A}_0^{-1}}= - {\bf A}_0^{-1}\pa{ \partial_x {\bf A}_0} {\bf A}_0^{-1},\]
which can be seen from differentiating the identities ${\bf A}_\delta^{-1} {\bf A}_\delta=I_n$ and ${\bf A}_0^{-1} {\bf A}_0=I_n$.  We also have (see \eqref{eq:2p})
\be\label{eqAderv}
 \partial_x {\bf A}_0=\partial_x {\bf A}_\delta= {\begin{pmatrix}\beta_0&0&\dots&0\\0&\beta_1&\dots&0    \\\vdots &\vdots &\ddots&\vdots  \\0 &0 &\dots&\beta_n  \end{pmatrix}}.
 \ee
Finally, we use \eqref{eq:resoldiff},  \eqref{eq:Adiff},  \eqref{eq:A_0norm}, \eqref{eq:prrul}, and  \eqref{eqAderv} to obtain the bound
\[\norm{ \partial_x({{\bf A}_\delta^{-1}- {\bf A}_0^{-1}})}\le K\delta\]
for some $\delta$-independent constant $K$.

Using this observation in \eqref{eq:prrulF} yields the estimate 
\be
\max_{x\in[0,1]}\abs{\partial_x({F_\delta(x)-F_0(x)})}\le \hat K \delta
\ee
for some $\hat K>0$. Thus, using \eqref{eq:normdif} as well,
\begin{multline*}
\max_{u\in S}\abs{G'(u)}\le \hat K\delta \max_{u\in S}\abs{\frac{{\pa{\beta_0\pa{u+y_1}+\mu}\pa{\beta_n \pa{u+y_1}+\mu+\omega_n}}}{ \beta_0\beta_n\mu\pa{u+y_1-y_2} }}\\ +\check K\delta\max_{u\in S}\abs{\partial_u\frac{{\pa{\beta_0\pa{u+y_1}+\mu}\pa{\beta_n \pa{u+y_1}+\mu+\omega_n}}}{ \beta_0\beta_n\mu\pa{u+y_1-y_2} }},
\end{multline*}
for some $\check K>0$. Since \[\max_{u\in S}\abs{\frac1{u+y_1-y_2}}\le 2\delta^{-1/3}\] for $\delta$ sufficiently small, after taking the derivative on the right hand side of the previous equation, we obtain the bound
\be
\max_{u\in S}\abs{G'(u)}\le \tilde K\delta^{1/3}
\ee
for some $\tilde K>0$. 
Hence, by the fundamental theorem of calculus, 
\be
 \abs{G(u)-G(v)}\le  \alpha\abs{u-v} \mbox{ for all } u,v \in S, \quad \alpha=\tilde K \delta^{1/3},
 \ee 
 and so \ref{it:b} has been established as well, and the solution for $I^*$ exists and is unique.  \end{proof}

\subsection{Linear Stability Analysis for Endemic Equilibria}\label{sec5}
We will prove the following  results:

\vspace{-.2cm}

\begin{thm}\label{thm:endn}
The endemic equilibrium for \eqref{eq:3c} is asymptotically stable for \[\pa{\omega_n+\mu}\pa{\mu+r}<{\beta_0\omega_n+\beta_n\mu},\] provided  $\delta$ is sufficiently small.  
\end{thm}

\begin{rem}
We note that for system \eqref{eq:3ca} $\omega_n=0$, in which case the condition above becomes $r+\mu<\beta_n$. Furthermore, given $\delta << 1$, for both the general model \eqref{eq:3c} and its specific cases \eqref{eq:3ca} and \eqref{eq:3aa}, the endemic equilibrium is asymptotically stable whenever it exists.
\end{rem}

\begin{rem}
Combining Theorems \ref{thm:DFE_ka}, \ref{thm:1}, and \ref{thm:endn}, we see that, when $\delta$ is sufficiently small, the stability condition for the DFE is consistent with the inexistence condition for the endemic equilibrium, and vice versa, for both the general model \eqref{eq:3c} and the specific cases \eqref{eq:3ca} and \eqref{eq:3aa}. I.e., if $\delta < < 1$ and the DFE is stable, no realistic (between $0$ and $1$) endemic equilibrium exists. If however the DFE is unstable, the endemic equilibrium exists and (by Theorem \ref{thm:endn}) is asymptotically stable. \end{rem}

\vspace{-.2cm}

\begin{proof}[Proof of Theorem \ref{thm:endn}]
We begin with the case when $\delta = 0$. Here  $S_i^*=0$ for $i\neq0,n$. Hence the Jacobian matrix ${\bf J}$ in \eqref{eq:2aa} reduces to
\[{\bf J} = 
{\begin{pmatrix}d_0(I^{*})&\omega_1&\dots   &\omega_{k-1}&\omega_k&\dots &\omega_{n-1}&\omega_{n}&r-\beta_0 S_0^{*} \\0 &d_1(I^{*})&\dots&0&0 &\dots &0 &0& 0 \\0& 0 &\dots &0&0 &\dots&0 &0 & 0  \\\vdots &\vdots &\ddots&\vdots&\vdots&\ddots&\vdots &\vdots & \vdots  \\0 &0&\dots & 0 &d_{k}(I^{*}) &\dots&0&0 & 0  \\\vdots &\vdots &\ddots&\vdots&\vdots&\ddots &\vdots  &\vdots &\vdots  \\0 &0 &\dots&0&0&\dots&d_{n-1}(I^{*})&0 & 0 \\0 &0 &\dots&0&0&\dots& 0 & d_n(I^{*}) &-\beta_n S_n^{*} \\\beta_0 I^{*} &\beta_1I^{*} &\dots   &\beta_{k-1}I^{*} &\beta_kI^{*} &\dots&\beta_{n-1}I^{*}&\beta_{n}I^{*} & 0  \end{pmatrix}},\]
where \eqref{eq:endco} has been used.

Evaluating $p(z)=\det\pa{{\bf J}-z\mathds{1}}$ by the second row, the resulting minor by the second row, etc., yields that $p(z)$ is equal to
\[\prod_{i=1}^{n-1}(-\mu-\beta_iI^* - w_i -z)\cdot \det{\begin{pmatrix}-\mu-\beta_0I^*-z&\omega_n&r-\beta_0 S_o^* \\0 &-\mu-\omega_n-\beta_n I^*-z&-\beta_n S_n^*\\\beta_0I^*&\beta_{n}I^*&-z\end{pmatrix}}.\]
The sum of all three columns in the above matrix is equal to $(-\mu-z)\begin{pmatrix}1&1&1\end{pmatrix}$. Hence, the determinant above coincides with 
\[(-\mu-z)\, \det{\begin{pmatrix}-\mu-\beta_0I^*-z&\omega_n&r-\beta_0 S_0^* \\0 &-\mu-\omega_n-\beta_n I^*-z&-\beta_n S_n^*\\1&1&1\end{pmatrix}}.\]
We now compute
\[ \det{\begin{pmatrix}-\mu-\beta_0I^*-z&\omega_n&r-\beta_0 S_0^* \\0 &-\mu-\omega_n-\beta_n I^*-z&-\beta_n S_n^*\\1&1&1\end{pmatrix}}=z^2+a(\omega_n)z+b(\omega_n),\]
where, using \eqref{eq:endco}, 
\[
\begin{aligned}
a(\omega_n)&=\omega_n+\pa{r-\beta_0 S_o^*-\beta_n S_n^*+\mu}+\pa{\beta_0I^*+\mu+\beta_n I^*}\\ &=\omega_n+\pa{\mu+\beta_0I^*+\beta_n I^*};\\
 b(\omega_n)&=\pa{\mu+\beta_0I^*+r-\beta_0 S_0^*-\beta_n S_n^*}\omega_n\\ 
&\quad +\pa{\mu+\beta_nI^*}\pa{\mu+\beta_0I^*+r-\beta_0 S_o^*}-\pa{\mu+\beta_0I^*}\beta_n S_n^*\\ &= \beta_0I^*\omega_n+\pa{\mu+\beta_nI^*}\pa{\beta_0I^*+\beta_n S_n^*}-\pa{\mu+\beta_0I^*}\beta_n S_n^*.
\end{aligned}
\]
Clearly, $\mu+\beta_0I^*+\beta_n I^*>0$. On the other hand, as  $\beta_n>\beta_0$, we get
\begin{multline*}
\pa{\mu+\beta_nI^*}\pa{\beta_0I^*+\beta_n S_n^*}-\pa{\mu+\beta_0I^*}\beta_n S_n^*\\ >\pa{\mu+\beta_nI^*}\beta_n S_n^*-\pa{\mu+\beta_0I^*}\beta_n S_n^*\ge0.\end{multline*}

Thus, the roots of $z^2+a(\omega_n)z+b(\omega_n)$ have negative real parts for all values of $\omega_n \ge 0$. Hence, by standard perturbation theory (see, e.g., \cite{Kato}), the eigenvalues of  $V$ have negative real parts for $\delta$ sufficiently small and the (realistic) endemic equilibrium is asymptotically stable by the Hartman-Grobman theorem for system \eqref{eq:3c}, and is therefore asymptotically stable for \eqref{eq:3ca} and \eqref{eq:3aa} as well. The assertion now follows from Theorem \ref{thm:1}. 
\end{proof}

\begin{rem}
It is also useful to consider the case $\omega = \delta$ (in which the rate to $S_0$ for any compartment $S_k$, $k > 0$, becomes $\delta p_k$). In this case, for sufficiently small $\delta$, the condition for the existence and stability of the endemic equilibrium for \eqref{eq:3c} and \eqref{eq:3aa} reduces to $\beta_{n} > r + \mu$ (and in particular becomes equivalent to that of \eqref{eq:3ca}). 
Additionally, the bounds are slightly altered: The relation \eqref{eq:Adiff} becomes $\norm{{\bf A}_\delta - {\bf A}_0}\le \delta\sqrt{2(n+1)}$, while the relation in \eqref{eq:A_0norm} becomes $\norm{{\bf A}_0^{-1}}\le \frac{1}{\beta_0I^*+\mu}$. The remaining results are unchanged.

\end{rem}
\section{Epidemiological Findings and Numerical Analysis} \label{Epidem}

In this section, we consider numerically-generated dynamical behavior associated with \eqref{eq:3c} and provide a comparison between the latter  and publicly available data for pertussis incidence in Canada, \cite{Public}. In order to work with relatively constant, homogenous demography rates, we restrict our population to consist only of individuals in the $10$ years old and above age group. Our simulation encompasses a four-compartmental ($n = 2$) system associated with \eqref{eq:3c}.  Values for the parameters $\beta_2, r, \mu, \delta$ are initially set to coincide with those in \cite{OSBPR}, with the exception of $\mu$, which is altered to fit the age restriction. Original values for $ \beta_1$ are estimated from \cite{Schwartz}, by scaling $\beta_2$ with a constant $\gamma \approx 0.635$ derived from the percent efficacy of pertussis vaccines after $5$ years, respectively (since we have fixed the total time needed to pass from $S_0$ to $S_2$ at $10$ years as per \cite{OSBPR}, $\delta = \frac{1}{5}$, and so individuals reach $S_1$ in $5$ years after vaccination). Finally, we select the vaccination scheme where the coverage coefficients $p_1, p_2$ equal $p_1 = 0.2$ and $p_2 = 0.62$.

Our simulation is arranged to run from $1801$ to $2020$. Time-series pertussis data, see Figure \ref{fig4a}, is retrieved from the database of the Public Health Agency of Canada, \cite{Public}, where for each year between $1991$ and $2020$ we have divided the sum total of pertussis cases among individuals ten and older by the population of Canada in that year.

It should be noted that $\beta_1$, among other parameter values, is derived from clinical data sources rather than a population data source. Disparities when extrapolating to the population scale may skew results. Furthermore, our simulated population is not a closed one by definition, and our models consider only those susceptible-infectious interactions that occur within the population. As a result, the simulations below are primarily theoretical representations of the model's behavior at relevant pertussis-based parameter tuples. Prior to contrasting our numerical results with real-world ones, we first consider the effect of varying the rates $\beta_0$, $\delta$, and $\omega$.

We first vary $\beta_0$, having set $\delta$ as before and allowed $\omega$ to be relatively large ($\omega = 20$ years$^{-1}$). This yields Figure \ref{fig2a} below. As illustrated on the latter, unless the initial vaccine effectiveness is very high, the system's dynamics will approach an endemic equilibrium even for high vaccination rates. (In particular, the system attains the DFE only if $\beta_0 < 9$, equivalent to an initial vaccine effectiveness of $ > 96.5 \%$.)  
 
We now set $\beta_0 = 9$ and vary $\omega$, which encompasses both the rate at which individual vaccination produce immunity and the rate at which individuals falling under the vaccine coverage are immunized. As shown on Figure \ref{fig2b}  prevalence peaks at an overwhelming $I = 0.37$ within a month for all $\omega \in [5, 30]$ considered. For lower $\omega \le 19$, prevalence steadies between a proportion of $0.003$ and a proportion of $0.03$ infectious individuals in the population. When $\omega$ is approximately $19$ years$^{-1}$, the solution's trajectory no longer approaches an endemic equilibrium, but rather attains the disease-free state (when prevalence is rounded to the nearest $10^{-10}$), see Figure \ref{fig2c}. Fixing $\omega$ at $20$ years$^{-1}$, we now vary $\delta$. On Figures \ref{fig3a}--\ref{fig3b} below, we illustrate the eventual behavior of the disease for $\delta \in [0.002, 0.266]$ and $\delta \in [0.3, 5.1]$. An apparent transcritical bifurcation occurs at $\delta \approx 0.21$, with the disease-free equilibrium losing stability.

\begin{figure}[h!]
\begin{center}
\hfill 
\subfigure{\label{fig2a}\includegraphics[width=7.3cm]{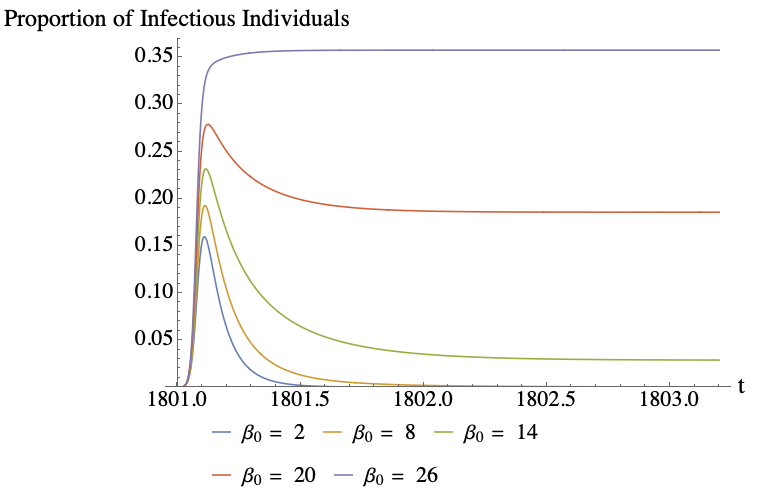}}
\subfigure{\label{fig2b}\includegraphics[width=7.9cm]{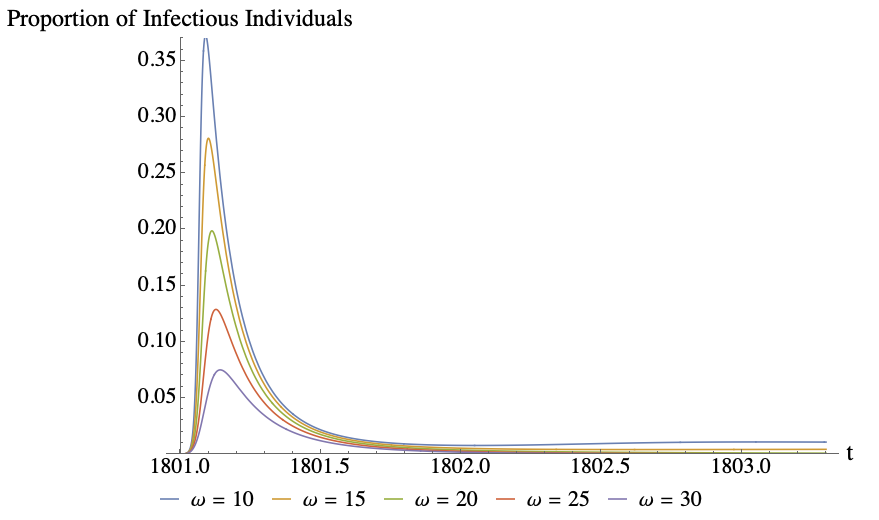}}
\subfigure{\label{fig2c}\includegraphics[width=7.9cm]{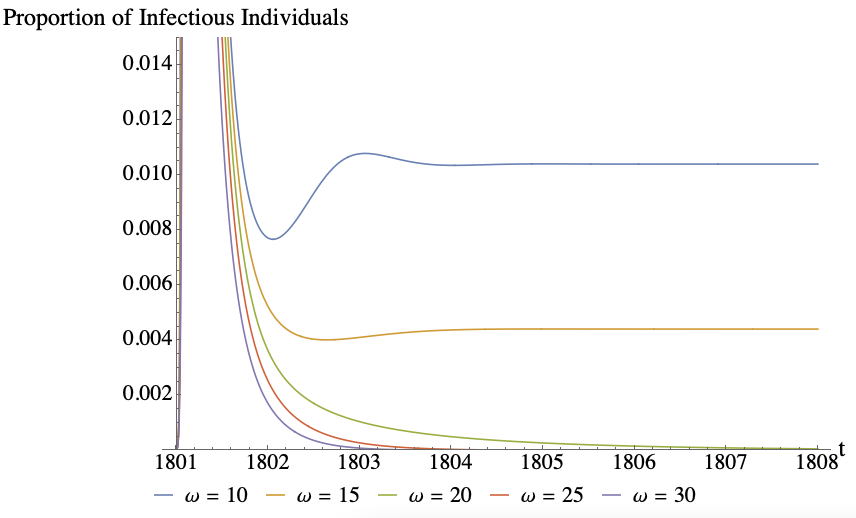}}

\end{center}
\caption{(a) Disease prevalence between 1801 and 1803 for various values of $\beta_0$. (b) Disease prevalence between 1801 and 1803 for various values of $\omega$. (c) Eventual behavior of the disease for the same values of $\omega$.  All figures were produced with Wolfram Mathematica.}
\end{figure}
No Hopf bifurcations and regions of (positive) endemic equilibrium instability are visible for low-to-high $\delta \in [\frac{1}{500}, 500]$ (we recall our theoretical result that none exist for $\delta$ sufficiently small in Section \ref{sec5}). Explicitly calculating the characteristic polynomial of the Jacobian matrix $J$ in \eqref{eq:2aa} at the endemic equilibrium (for all $\delta > 0$) at small values of $n$, we verified that the coefficients of the characteristic polynomial exhibit no sign changes given the relation $\beta_0 \le \beta_1 \le \ldots \le \beta_n$. Thus, for small $n$, an endemic equilibrium is stable for all positive $\delta$ whenever it exists. We conjecture that this behavior holds true for general values of $n$.  
\begin{figure}[h!]

\hfill 
\subfigure{\label{fig3a}\includegraphics[width=7cm]{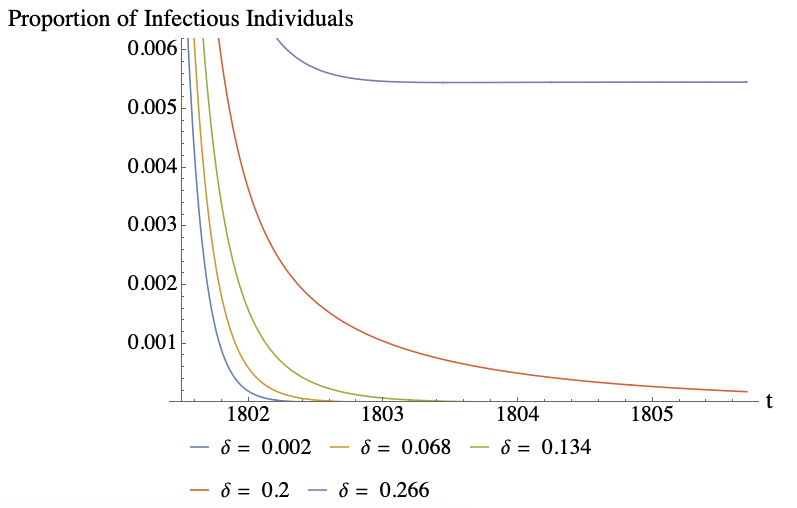}}
\subfigure{\label{fig3b}\includegraphics[width=7cm]{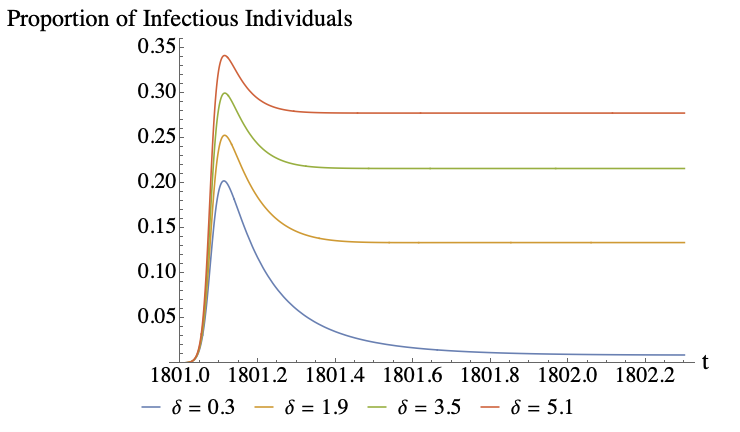}}

\caption{(a) Later behavior of the disease for $\delta \in [0.002, 0.266]$. (b) Later behavior of the disease for $\delta \in [0.3, 5.1]$.}\label{fig3}
\end{figure}
\begin{figure}[h!]
\hfill 
\subfigure{\label{fig4a}\includegraphics[width=6.9cm]{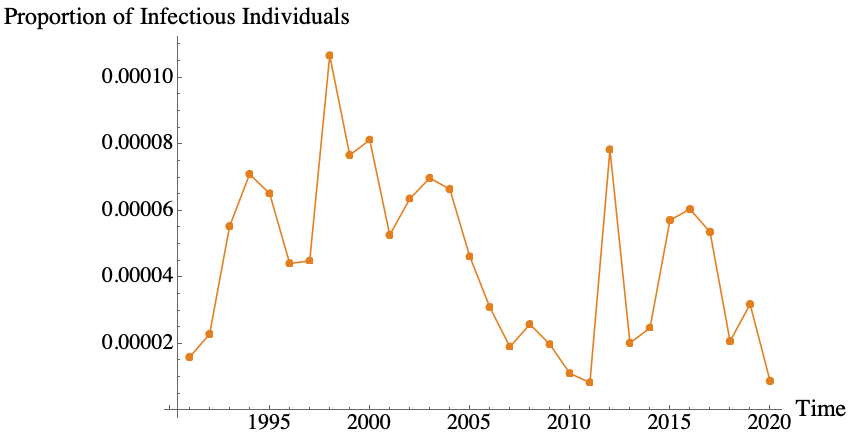}}
\subfigure{\label{fig4b}\includegraphics[width=6.9cm]{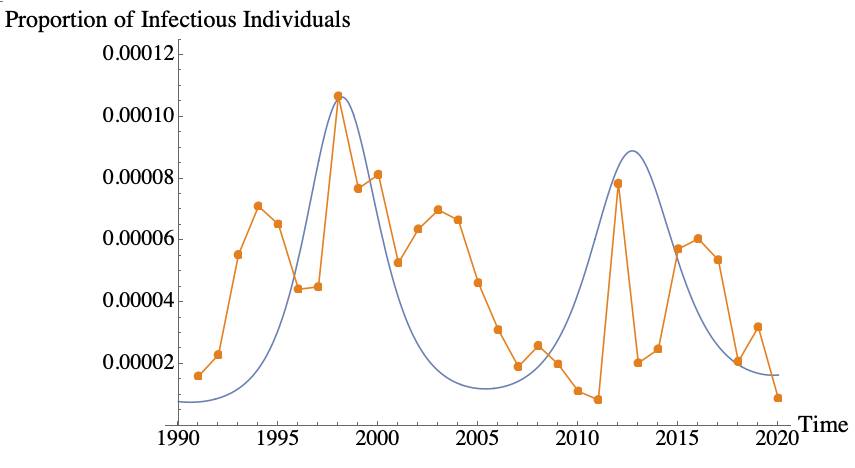}}
\subfigure{\label{fig4c}\includegraphics[width = 9cm]{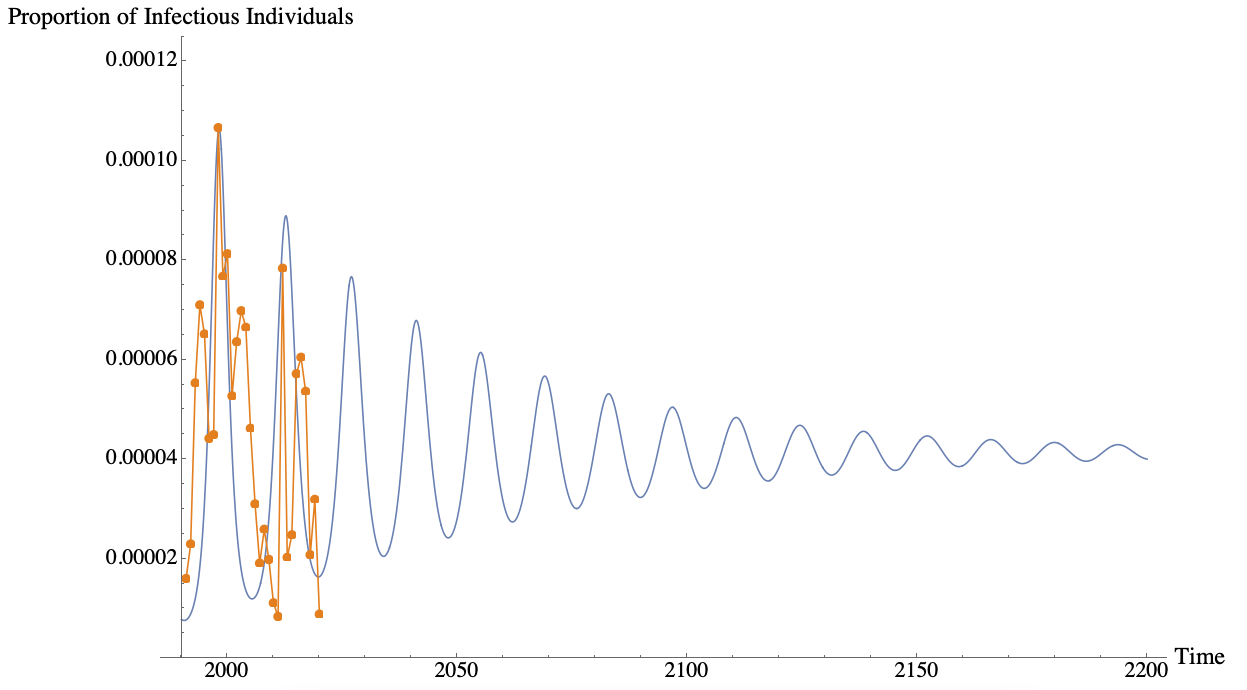}}\caption{(a) The time-series data for pertussis among individuals ten and older in Canada between 1990 and 2020. (b) Our numerical simulation (in blue) fitted to the time-series data. (c) Our numerical simulation extended for another two hundred years.}
\end{figure}

We now consider values for the parameters that fit our numerical simulation to the aforementioned time-series data. Though for lower realistic values of $\delta, \omega$ one sees oscillations in disease prevalence similar to those exhibited in the time-series, the system's dynamics when $\omega > 1, \delta > 1/10$ either approach the DFE or tend towards an endemic equilibrium significantly greater than the values of $I$ reflected in the time-series data (e.g., $I(1806) \approx 0.0041$ in Figure \ref{fig2c} above compared to $I(2011) \approx 0.00008$ in the time-series data). Furthermore, prevalence in the first few months of the epidemic is unrealistically high. 

To able to fit the numerics to the time-series data, we alter the above parameters, many of which are forced to become biologically-unrealistic. $\omega$, $\mu$, and $\delta$, for example, both fall very low (e.g., $\omega \approx 0.0378$), while others (i.e., $\beta_0, \beta_1, \beta_2$) are increased by a factor of $1.8$ from their previous values. This yields Figure \ref{fig4b} above. In this case, continuing the simulation further yields that oscillations gradually dampen until the system becomes arbitrarily close to an endemic equilibrium, with $I^{*} \approx 0.0000414$, or $0.00414\%$ of the total population.

\section{Conclusion}\label{sec7}

In this paper, we develop a perturbative approach to  determine the location, uniqueness, and stability for the endemic equilibria of a waning immunity model \eqref{eq:3c} as well as its particular cases \eqref{eq:3ca} and \eqref{eq:3aa}. 

We also identify an optimal booster vaccine distribution scheme within our model framework, evaluating analytically which of the vaccination methods featured is able to both minimize $R_0$ and maximize time between successive immunizations. 

In particular, $R_0$ in the case when all compartments other than $S_n$ have access to immunizations is independent of the values of $\omega$ and $\{p_i\}$, and is equal to $\frac{\beta_{n}}{r + \mu}$ regardless of whether any compartments are vaccinated. It should be noted here that, for many diseases, transmission rates for fully-susceptible individuals are large enough to far exceed the recovery and birth rates together (the pertussis parameters mentioned in \cite{OSBPR} and \cite{Lavine}, for example, exhibit this relation). This leads to the $R_0$ for \eqref{eq:3ca} exceeding $1$, and the disease-free equilibrium being unstable for this vaccination scheme. On the other hand, when only members of $S_n$ are vaccinated, the $R_0$ decreases to an $\omega_{n}$-dependent value, which could conceivably be brought below $1$ for sufficiently high $\omega$ and $p_n$. 

The implication is that any scheme of the form \eqref{eq:3ca} is not an effective choice for vaccination, and that booster vaccination methods should devote resources to vaccinating fully-susceptible individuals.

When considering possible applications of this research, there are two main directions in which it can be applied: Methods such as ours work to extend the range of infectious disease models that can be explicitly analyzed, by enabling the analytical and stable numerical study of endemic behavior in many-compartment waning immunity models, among others. Additionally, the perturbation theory utilized is sufficiently versatile and robust to be applicable to a a more general setting than the one used in this work, whose primary goal is to demonstrate the method's efficiency. 

We, however, expect that this method in general may break down when the rate of waning immunity becomes large.

\appendix
\section{Completing the argument for Theorem \ref{thm:inf_free}}\label{sec:app}
Denoting ${\bf  S}(t):=(S_0(t),\ldots,S_n(t))$, ${\bf  S}_o:=(s_0,\ldots,s_n)$, we have 
\be
\label{eq:vecS}
{\bf  S}^\intercal(t)=\exp\pa{{\bf J}t}({\bf  S}_o)^\intercal+\mu\exp\pa{{\bf J}t}\int_0^t\exp\pa{-{\bf J}s}{\bf e}_n^\intercal ds,
\ee

Since ${\bf J}$ is a lower triangular matrix, its eigenvalues are its diagonal entries, i.e., its spectrum $\sigma({\bf J})=\{-\delta-\mu,-\mu\}$. In particular, the asymptotic behavior of $\exp\pa{{\bf J} t}$ as $t\to\infty$ is completely determined by the largest, simple eigenvalue $-\mu$ with the eigenprojection
 \[{\bf P}_\mu=  {\begin{pmatrix}0&0&\dots   &0  \\\vdots &\vdots&\ddots &\vdots  \\0 &0&\dots &0   \\1 &1 &\dots&1 \end{pmatrix}},\]

\noindent in the sense that 
 \be\label{eq:limfr}
 \lim_{t\to\infty}\mathrm{e}^{\kappa t}\norm{\pa{\mathrm{e}^{\mu t}\exp\pa{{\bf J} t}-P_\mu}}=0
 \ee
 for any $\kappa<\delta$. 
  Using \eqref{eq:vecS}, $\pa{\mathds{1}-{\bf P}_\mu }{\bf P}_\mu=0$, and ${\bf P}_\mu\exp\pa{{\bf J} t}=\mathrm{e}^{-\mu t}{\bf P}_\mu$, we get 

 \be
 \begin{aligned}
{\bf  S}^\intercal(t)-{\bf e}_n^\intercal 
& = {\bf P}_\mu \pa{{\bf  S}^\intercal(t)-{\bf e}_n^\intercal}+\pa{\mathds{1} - {\bf P}_\mu } \pa{{\bf  S}^\intercal(t)-{\bf e}_n^\intercal}\\
& =\mathrm{e}^{-\mu t} {\bf P}_\mu{\bf  S}^\intercal_o-\mathrm{e}^{-\mu t}{\bf e}_n^\intercal+\pa{\mathds{1} - {\bf P}_\mu} \exp\pa{{\bf J} t}{\bf  S}_o^\intercal.
\end{aligned}
 \ee 
 
 Hence, decomposing
 \[\pa{\mathds{1}-{\bf P}_\mu } \exp\pa{ {\bf J} t}{\bf  S}_o^\intercal=\mathrm{e}^{-\mu t}\pa{\mathds{1}-{\bf P}_\mu } \pa{\mathrm{e}^{\mu t}\exp\pa{{\bf J} t}-{\bf P}_\mu}{\bf  S}_o^\intercal\]
and using \eqref{eq:limfr} and $\norm{\mathds{1}-{\bf P}_\mu }=\sqrt{n + 1}$, we can deduce 
\be\label{eq:ret}
 \begin{aligned}
\lim_{t\to\infty}&\norm{\mathrm{e}^{\kappa t}\pa{{\bf  S}^\intercal(t)-{\bf e}_n^\intercal}}\le \lim_{t\to\infty}\mathrm{e}^{-\pa{\mu-\kappa} t}\\&\times\pa{\norm{{\bf P}_\mu{\bf  S}_o^\intercal}+1+\sqrt{n + 1}\norm{\mathrm{e}^{\mu t}\exp\pa{{\bf J} t} - {\bf P}_\mu}\norm{{\bf S}_o^\intercal}}=0
\end{aligned}
\ee
for $\kappa<\mu$.  

This is expected, as ${\bf e}_n$ is a unique equilibrium for \eqref{eq:2free}. It can also be seen from \eqref{eq:ret} that the decay to the equilibrium occurs exponentially fast. 
\qed

\vspace{1cm}

\acknowledgement{
We thank Dr. Lauren Childs for originally suggesting the problem, as well as for her very helpful discussions and feedback.}

\printbibliography

\end{document}